\documentclass[reqno,12pt]{amsart}
\numberwithin{equation}{section} \parskip1mm
\usepackage{amsmath}
\usepackage{amsthm}
\usepackage{amssymb}
\usepackage{epsfig}
\usepackage{psfrag}
\usepackage{graphicx}
\usepackage{graphpap,latexsym}
\usepackage{color}
\usepackage{version}	
\usepackage{amssymb,mathrsfs,enumerate}

\newcommand{\R}{\Rz}
\newcommand{\gciapo}{\widehat g}
\newcommand{\betaciapo}{\widehat \beta}
\newcommand{\N}{\Nz}
\newcommand{\Rz}{\mathbb{R}}
\newcommand{\RR}{\mathbb{R}}
\newcommand{\RN}{\mathbb{R}^N}

\newcommand{\RdN}{\mathbb{R}^{2N}}

\newcommand{\Nz}{\mathbb{N}}

\newcommand{\loc}{\text{\rm loc}}

\DeclareMathOperator{\sign}{sign}
\newcommand{\quext}{\quad\text}

\newcommand{\io}{\int_\Omega}

\newcommand{\barO}{\overline{\Omega}}

\newcommand{\dix}{{\rm d} x}

\newcommand{\dit}{{\rm d} t}

\newcommand{\diy}{{\rm d} y}

\newcommand{\dir}{{\rm d} r}

\newcommand{\dx}{{\rm d} x}
\newcommand{\dy}{{\rm d} y}

\newcommand{\Asig}{\mathfrak{A}_\sigma}

\newcommand{\As}{\mathfrak{A}_s}

\newtheorem{corollary}{Corollary}[section]
\newtheorem{propo}{Proposition}[section]
\newtheorem{remark}{Remark}[section]
\newtheorem{lemma}{Lemma}[section]
\newtheorem{defi}{Definition}[section]
\newtheorem{theorem}{Theorem}

\newcommand{\LR}{L^2(\RN)}
\newcommand{\LO}{H_0} 

\newcommand{\ichi}{i}
\newcommand{\nii}{iii}
\newcommand{\san}{ii}
\newcommand{\shi}{iv}

\renewcommand{\d}{\hbox{\rm d}}
\newcommand{\vep}{\varepsilon}
\newtheorem{claim}{Claim}[section]

\newcommand{\lista}{\begin{list}{}{\setlength{\leftmargin}{0.6in}\setlength{\labelwidth}{1.5in}\setlength{\labelsep}{0.2in}}}
\newcommand{\finelista}{\end{list}} 

\newcommand{\Xsigz}{\mathcal{X}_{\sigma 0}}
\newcommand{\Xsz}{\mathcal{X}_{s 0}}
\newcommand{\Xp}{X_p^\sigma}

\newcommand{\Ds}{(-\Delta)^s}

\newcommand{\Dsig}{(-\Delta)^\sigma}

\newcommand{\Dsigm}{(-\Delta)^{\sigma/2}}

\newcommand{\Esi}{\mathbb{E}_{\sigma}}
\newcommand{\J}{\Esi}
\newcommand{\tJ}{\mathbb{\tilde E}_{\sigma}}
\newcommand{\Id}{\mathrm{Id}}

\newcommand{\LS}{\L S}
\newcommand{\perogni}{\forall\,}

\DeclareMathOperator{\deriv}{d}

\newcommand{\ddt}{\frac{\deriv\!{}}{\dit}}

\newcommand{\BBB}{\color{blue}}
\newcommand{\SSS}{\color{black}}

\begin{document}

\title[Fractional Cahn-Hilliard equation]{Convergence of solutions for the fractional Cahn-Hilliard system}

\pagestyle{myheadings}

\author{Goro Akagi}
\address{Mathematical Institute, Tohoku University}
\author{Giulio Schimperna}
\address{Dipartimento di Matematica ``F. Casorati'', via Ferrata 5, I--27100 Pavia, Italy}
\author{Antonio Segatti}
\address{Dipartimento di Matematica ``F. Casorati'', via Ferrata 5, I--27100 Pavia, Italy}
\email{akagi@m.tohoku.ac.jp}
\email{giusch04@unipv.it}
\email{antonio.segatti@unipv.it}
\urladdr{}
\urladdr{http://www-dimat.unipv.it/giulio}
\urladdr{http://www-dimat.unipv.it/segatti}
\date{\today}


\maketitle

 \begin{abstract}
  This paper deals with the Cauchy-Dirichlet problem for the fractional Cahn-Hilliard equation. The main results consist of global (in time) existence of weak solutions, 
  characterization of parabolic smoothing effects (implying under proper condition eventual boundedness of trajectories), and convergence of each solution to a (single) equilibrium. 
  In particular, to prove the convergence result, a variant of the so-called \L ojasiewicz-Simon inequality is provided for the fractional Dirichlet Laplacian 
  and (possibly) non-analytic (but $C^1$) nonlinearities.
 \end{abstract}


\section{Introduction}
\label{sec:intro}

The present paper is concerned with the long-time behavior of solutions to the Cauchy-Dirichlet problem for a fractional version of the Cahn-Hilliard equation. Let $\Omega$ be a bounded domain of $\mathbb{R}^N$ with smooth boundary $\partial \Omega$. For $s, \sigma \in (0,1)$, let us consider
\begin{alignat}{4}\label{eq:fCH}
  & u_t  + \Ds w = 0 \ &\hbox{ in }&  \Omega \times (0,+\infty), \\
  & w = \Dsig u + g(u) \ &\hbox{ in }& \Omega\times (0,+\infty),  \label{eq:chem_pot}\\
  & u|_{t=0} = u_0 \ &\hbox{ in }& \Omega, \label{eq:iniz_bc}\\
  & u = w = 0 \ &\hbox{ in }& \RN\setminus \Omega,\label{eq:bc} 
\end{alignat}
where $\Ds$ (and $\Dsig$) is the so-called \emph{fractional Laplacian}
defined by
\begin{align*}
 \Ds u (x) &= C(N,s) \mbox{ p.v.} \int_{\RN} \dfrac{u(x)-u(y)}{|x-y|^{N+2s}} \, \d y\\
 &= - \dfrac{C(N,s)}2 \int_{\RN} \dfrac{u(x + h) - 2 u(x) + u(x-h)}{|h|^{N+2s}} \, \d h.
\end{align*}
Here p.v.~stands for the Cauchy principal value and $C(N,s)$ is a positive constant determined by $N$ and $s$ only (see, e.g.,~\cite{Dine_Pala_Vald}), 
and hereafter, it will be simply denoted by $C_s$. Moreover, $g$ is the derivative of a 
smooth potential $\gciapo : \R \to \R$ (a typical choice of $\widehat g$ is a double-well potential of the form,
$$
\gciapo (s) = \dfrac{|s|^m}{m} - \dfrac{s^2}2 \quad \mbox{ for } \ s \in \R
$$
and for some $m \geq 4$).

Various types of nonlocal Cahn-Hilliard equation have been proposed and intensively studied by many authors. Among the many
contributions, we may quote with no claim of completeness \cite{ABG,ASSe1,BaHa,GGG,MR} (see also the 
references therein for a more comprehensive bibiliography). Recent applications of the equation mostly
refer to complex (two-phase) fluids \cite{DPG,FGG}, and stochastic models \cite{Cor}. It is worth noting that, in most
of the quoted works, the nonlocal operator is obtained through the convolution with a kernel that is (at least)
summable over $\RN$. This gives rise to a different functional setting compared to here in view of the fact that,
if the kernel is summable, the solution loses the smoothing properties that are proper of parabolic equations. 
Up to our knowledge, the only contributions where the nonlocal operator is obtained by convolution with a singular kernel
are \cite{ABG} and our former work \cite{ASSe1}, where Cahn-Hilliard equations accounting for (different types of)
fractional Laplace operators are studied. 

We recall that the fractional Laplacian may be defined in different ways:
in terms of singular integrals as above; by Fourier transform setting 
$\Ds u = \mathcal F^{-1} \left[ |\xi|^{2s} \mathcal F[u] \right]$ (where $\mathcal F$ and $\mathcal F^{-1}$ denote Fourier and inverse Fourier
transforms, respectively); by extension methods; using the heat semigroup; in a probabilistic way (as a generator of Levy processes).
All these definitions are equivalent to each other once one considers $\Ds$ on the whole space $\RN$. On the other hand, 
when working, as here, on bounded domains some more care is required. Indeed, 
one can formulate the fractional Laplacian $\Ds$ equipped with the solid Dirichlet boundary condition, $u=0$ in $\mathbb R^N \setminus \Omega$ 
(it will be referred as \emph{fractional Dirichlet Laplacian} for short) in a variational fashion by means of the weak form,
$$
\left\langle \Ds u , v \right\rangle_{\Xsz}
= \dfrac{C_s}2 \iint_{\RdN} \dfrac{\left(u(x)-u(y)\right)\left(v(x)-v(y)\right)}{|x-y|^{N+2s}} \, \dx \, \dy
\quad \mbox{ for } u,v \in \Xsz
$$
where $\Xsz$ is a Hilbert space given by
\begin{align}
 \Xsz := \bigg\{ v\in L^2(\RN) \colon & v = 0 \ \mbox{ a.e.~in } \RN \setminus \Omega \nonumber \\
 & \mbox{and } \ (x,y)\mapsto \dfrac{|v(x)-v(y)|^2}{|x-y|^{N+2s}} \in L^1(\RdN) \bigg\} \label{Xsz}
\end{align}
furnished with inner product,
$$
(u,v)_{\Xsz} = \int_\Omega u(x)v(x) \, \d x + \dfrac{C_s}2 \iint_{\RdN} \dfrac{\left(u(x)-u(y)\right)\left(v(x)-v(y)\right)}{|x-y|^{N+2s}} \, \dx \, \dy
$$
for $u,v \in \Xsz$. Then, we may introduce a weak form of 
$\Ds$ as an operator from $\Xsz$ to its dual space $\Xsz'$ (to be precise, 
we will denote it by $\As$ instead of $\Ds$ throughout this paper).
In such a setting, the fractional Cahn-Hilliard equation \eqref{eq:fCH} 
was first studied in~\cite{ASSe1}, where the well-posedness of \eqref{eq:fCH}--\eqref{eq:bc} and its singular 
limits for $s \to 0_+$ or $\sigma \to 0_+$ are treated. On the other hand, the long-time behavior of solutions has 
not yet been studied so far and, in particular, the convergence of trajectories
to $\omega$-limit sets as $t \to +\infty$ remains an open problem. Indeed, \eqref{eq:fCH}--\eqref{eq:bc} 
may have multiple equilibria, for the potential $\widehat g$ may be non-convex and have multiple (local) minimizers. So it is a delicate 
issue whether each solution converges to a single equilibrium as $t \to +\infty$, or, in other words, 
the $\omega$-limit set of each orbit is a singleton.

The problem of convergence of trajectories has been extensively studied in the case
of the classical Cahn-Hilliard system (namely, for $s = \sigma = 1$),
$$
\partial_t u = \Delta w, \quad w = - \Delta u + g(u) \ \mbox{ in } \Omega \times (0,+\infty),
$$
which can be combined into a single equation,
$$
\partial_t u = \Delta \left( - \Delta u + g(u) \right) \ \mbox{ in } \Omega \times (0,+\infty).
$$
Rybka \& Hoffman~\cite{RyHo} proved that each solution converges to a single equilibrium as $t \to +\infty$, provided that $g(\cdot)$ is a polynomial of order $n$ and $2 \leq n < 2^* := 2N/(N-2)_+$, 
using the so-called \emph{\L ojasiewicz-Simon inequality} for the elliptic operator $u \mapsto - \Delta u + g(u)$. The \L ojasiewicz-Simon (\LS)
inequality is an infinite-dimensional extension of the \L ojasiewicz inequality, which is a gradient inequality for analytic functions defined on an open set $U \subset \R^d$ (see~\cite{L,L2} and 
Proposition \ref{P:L} below). The \LS\ inequality was first introduced by L.~Simon~\cite{Simon83}
and has been subsequently applied to various PDEs with gradient(-like) structures by Haraux, Jendoubi, Chill and many other 
authors~(see, e.g., without any claim of completeness, \cite{Jen98,HJ98,Har00,FeiSim00,HJ01,HJK03,Chill03,Huang,CHJ09,HJ11,FeMa}).
A general form of the \L ojasiewicz-Simon 
inequality reads as follows: Let $E : X \to \mathbb R$ be a ``smooth'' functional defined on a Banach space $X$ and let $\phi$ be a critical 
point of $E$, i.e., $E'(\phi) = 0$ in the dual space $X^*$, where $E' : X \to X^*$ denotes the Fr\'echet derivative of $E$. 
Then there exist constants $\theta \in (0,1/2]$ and $\omega,\delta >0$ such that
$$
|E(v)-E(\phi)|^{1-\theta} \leq \omega \|E'(v)\|_{X^*} \quad \mbox{ for all
} v \in X \ \mbox{ satisfying } \ \|v-\phi\|_X < \delta.
$$
The above is, actually, a ``standard'' version of the inequality since there are, indeed,
several variants with different combinations of topologies. The \L ojasiewicz inequality 
(in finite dimensional spaces) essentially requires analyticity of functions. Therefore, smoothness of the energy $E$ is required. 
When moving to infinite dimensions and to PDE applications, however,
the analiticity of the energy $E$ turns out to be quite demanding. For instance, 
parabolic equations with power like nonlinearities with non integer exponents are 
ruled out of the theory. Therefore, several authors
tried to relax analyticity of the enery $E$. In particular, Feireisl \& Simondon~\cite{FeiSim00} established a \L S inequality for a $C^2$ functional of the form
$$
E(u) = \frac 1 2 \int_\Omega |\nabla u|^2 \, \d x + \int_\Omega h(u(x)) \, \d x \quad \mbox{ for } \ u \in X = H^1_0(\Omega),
$$
where $h : \R \to \R$ is bounded and of class $C^2$ over $\R$ and analytic on a subinterval $I = (0,\ell)$ with a 
singularity at the origin (then $E'(u)$ coincides with $-\Delta u + h'(u)$), and proved convergence as $t \to +\infty$ of bounded nonnegative solutions to positive equilibria for nonlinear diffusion equations.

Now, let us turn back to the fractional case and describe the main points of our contribution.
Firstly, \eqref{eq:fCH}--\eqref{eq:bc} cannot be combined into a single equation as 
happens for the classical model. Indeed, the fractional Laplacian is a nonlocal operator. 
In particular, $(-\Delta)^s u$ may have a tail at infinity even if 
$u$ has compact support. Consequently,
one cannot substitute \eqref{eq:chem_pot} into \eqref{eq:fCH}, since the value of $\Ds w(\cdot,t)$ is determined 
from all values of $w(\cdot,t)$ over the whole of $\R^N$, but equations \eqref{eq:fCH} and \eqref{eq:chem_pot} hold only on 
the domain $\Omega$. Secondly, to the best of authors' knowledge, the \L ojasiewicz-Simon inequality has not yet been proved to hold 
in the case of the fractional (Dirichlet) Laplacian, even when it is combined with analytic nonlinearities.

In the present paper, we shall actually extend the \L ojasiewicz-Simon inequality to the fractional (Dirichlet) Laplacian. Moreover, we shall apply it 
to prove convergence of solutions to \eqref{eq:fCH}--\eqref{eq:bc} for a (possibly) non-analytic potential $\widehat g$. On the other hand, proofs of the \L S inequality 
(for the Laplacian, see, e.g.,~\cite{RyHo} and~\cite{FeiSim00}) are often 
based on regularity theories for the elliptic equation $-\Delta u = f$ in $\Omega$, $u|_{\partial \Omega} = 0$ such as Schauder theory (i.e., $C^{2,\alpha}$-regularity) 
and $L^p$-theory of Agmon-Douglis-Nirenberg (i.e., $W^{2,p}$-regularity) as well as Hopf's lemma. However, the fractional (Dirichlet) Laplacian may not enjoy 
such regularity properties; indeed, the solution to the elliptic equation $\Ds u = 1$ in $\Omega$, $u = 0$ on $\R^N \setminus \Omega$ is at most 
of class $C^s(\overline\Omega)$ (see~\cite{RoSe12,RoSe13} for more details). In order to overcome such a difficulty,
we shall introduce a proper functional space $\Xp$ (which takes the role of the domain
of the fractional Laplacian seen as an unbounded linear operator on $L^p(\Omega)$).
Though we cannot properly identify $\Xp$ from the point of view of regularity, we will
be nevertheless able to prove that a \L ojasiewicz-Simon type inequality holds with respect
to its natural norm (see Theorem~\ref{T:LSI-bdd} below). 
The \L ojasiewicz-Simon inequality developed in the present paper can be also applied to verify 
convergence of bounded solutions for the fractional Allen-Cahn equation and fractional
nonlinear diffusion equations whose solutions preserve the sign of initial data (see~\cite{FeiSim00}).

Furthermore, the solution $u = u(x,t)$ of the fractional Cahn-Hilliard system \eqref{eq:fCH}--\eqref{eq:bc} may not preserve sign of initial data 
(like the classical Cahn-Hilliard system). Therefore we shall develop a \L S inequality in such a way as to cover (possibly) sign-changing
equilibria (cf.~\cite{FeiSim00}) as well as potential functions $\widehat g(\cdot)$ of any growth (more precisely, without imposing the Sobolev 
(sub)critical growth, cf.~\cite{RyHo}). In particular, we shall address ourselves only to \emph{bounded} 
solutions of \eqref{eq:fCH}--\eqref{eq:bc} when $\widehat g(\cdot)$ may not satisfy any growth condition. Therefore we shall also discuss (eventual) 
boundedness of solutions by observing a smoothing effect of solutions to \eqref{eq:fCH}--\eqref{eq:bc} for $s$, $\sigma$ belonging to a proper range. 
In the case of the (classical) Cahn-Hilliard equation, the problem of (eventual) boundedness of solutions has
been studied by several authors, starting from the pioneering work~\cite{Caf} dealing with the (more involved) case
where the equation is settled on the whole space $\R^N$. On the other hand, when one works on 
a smooth bounded domain $\Omega$, at least for $g$ smooth enough it is not difficult to see
that the answer to the boundedness question is generally positive, at least under the natural 
boundary conditions of (homogeneous) Dirichlet or Neumann type. Actually, in that case one can perform 
standard bootstrap arguments in view of the fact that (classical) Laplace operators can actually
be iterated. 

The organization of the present paper is as follows: Section \ref{sec:prelim} is a collection of 
basic notions on function spaces, functionals and operators 
as well as preliminary facts. In Section \ref{sec:main}, we state main results of the present papers. They consist of global existence of solutions and energy 
inequalities (see Theorem \ref{thm:exi} in \S \ref{sec:main}), smoothing effect and boundedness of global solutions (Theorems \ref{thm:reg}--\ref{thm:ener} 
and Proposition \ref{P:hoelder}), construction of non-empty $\omega$-limit sets (Lemma \ref{L:omega}), convergence of solutions (Theorem \ref{thm:Loj}) and 
the fractional \L ojasiewicz-Simon inequality (Theorem \ref{T:LSI-bdd}). In Section \ref{sec:para}, we shall briefly prove 
Theorems \ref{thm:exi}--\ref{thm:ener} and Proposition \ref{P:hoelder}. Section \ref{sec:omega} is devoted to a proof of 
Lemma \ref{L:omega}. In Section \ref{sec:Loj-2}, a variant of the \L ojasiewicz-Simon inequality will be established for 
the fractional Dirichlet Laplacian $\Dsig$, and then, Theorem \ref{thm:Loj} will be proved in Section \ref{S:pr-conv}. 
Appendix compensates several technical arguments and lemmas used in the main part of the paper.


\section{Preliminaries}
\label{sec:prelim}

In this section, we set up notation and recall some preliminary facts on fractional Laplace operators.

\subsection{Notation and function spaces}
Let $H$ be a Hilbert space identified with its dual space $H'$. For $M \subset H$, set
\begin{align*}
M^\bot &:= \left\{ f \in H \colon (f,u)_H = 0 \ \mbox{
for all } u \in M \right\},
\end{align*}
where $(\cdot,\cdot)_H$ is the inner product of $H$. For a Banach space $X$ and its dual space $X'$, we denote by $\langle \cdot, \cdot \rangle_X$ (or simply $\langle \cdot, \cdot \rangle$) a duality pairing between $X$ and $X'$.

Let $u = u(x,t): \Omega \times [0,\infty) \to \mathbb R$ be a function of space and time variables. Throughout the paper, for each $t \geq 0$ fixed, we simply denote by $u(t)$ the function $u(\cdot,t) : \Omega \to \mathbb R$ of space variable only.

Let $0 \le S < T \le + \infty$, $p\in [1,\infty)$ and let $X$ be a Banach space. When we simply write $f\in L^p_{\loc}(S,T;X)$, it actually means that $f\in L^p(S,T;X)$ if $T< \infty$ and $f\in L^p_{\loc}([S,\infty);X)$ if $T=\infty$. Furthermore, $C_w([S,T];X)$ denotes the space of continuous functions $t \in [S,T] \mapsto u(t) \in X$ in the weak topology of $X$.

For simplicity, the restriction $f|_\Omega$ of $f : \RN \to \R$ onto $\Omega$ is also simply denoted by $f$, if no confusion may arise. Moreover, for $p \in [1,\infty)$, the space
$$
L^p_0(\RN) := \left\{ u \in L^p(\RN) \colon u = 0 \ \mbox{ a.e.~in } \RN \setminus \Omega \right\}
$$
is identified with $L^p(\Omega)$, and we will use the same notation for functions in $L^p(\Omega)$ and $L^p_0(\RN)$.

Here and henceforth, 
$C$ denotes a constant independent of the elements of the corresponding space or set,
whose explicit value may vary from line to line. Let $\|\cdot\|_1$ and $\|\cdot\|_2$ be norms
on a vector space $X$.
We write $\|u\|_1 \lesssim \|u\|_2$ in the following sense: there is a constant $C \geq 0$ independent of $u$ such that
$$
\|u\|_1 \leq C \|u\|_2 \quad \mbox{ for all } \ u \in X.
$$
We also write $\|u_n\|_1 \lesssim \|u_n\|_2$ (for all $n \in \N$) in an analogous sense (with a constant $C$ independent of $n$).


\subsection{Fractional Sobolev and H\"older spaces}

For any measurable set $\mathcal O \subset \RN$, $s \in (0,1)$ and $p \in [1,\infty)$, we recall the \emph{fractional Sobolev spaces} defined by
$$
W^{s,p}(\mathcal O) := \{v \in L^p(\mathcal O) \colon (x,y) \mapsto |v(x)-v(y)|^p/|x-y|^{N+sp} \in L^1(\mathcal O \times \mathcal O)\}.
$$
Moreover, we write $H^s(\RN) := W^{s,2}(\RN)$. Let $[\,\cdot\,]_{W^{s,p}(\mathcal O)}$ be the \emph{Gagliardo seminorm} of $W^{s,p}(\mathcal O)$ given by
$$
[v]_{W^{s,p}(\mathcal O)} := \left(\iint_{\mathcal O \times \mathcal O} \dfrac{|v(x)-v(y)|^p}{|x-y|^{N+sp}} \, \dx \, \dy\right)^{1/p} \quad \mbox{ for } \ v \in W^{s,p}(\mathcal O).
$$
Then $W^{s,p}(\mathcal O)$ is furnished with the norm,
$$
\|\cdot\|_{W^{s,p}(\mathcal O)} := \|\cdot\|_{L^p(\mathcal O)} + [\,\cdot\,]_{W^{s,p}(\mathcal O)}.
$$

Furthermore, $C^\sigma(\mathcal O)$ stands for the space of  H\"older continuous functions with exponent $\sigma \in (0,1)$. In particular, if $\mathcal O$ is compact, then the norm $\|\cdot\|_{C^\sigma(\mathcal O)}$ is defined by
\begin{equation}\label{hoelder-norm}
 \|w\|_{C^\sigma(\mathcal O)} := \sup_{x \in\mathcal O} |w(x)| + \sup_{\substack{x,y \in \mathcal O\\x\neq y}} \dfrac{|w(x)-w(y)|}{|x-y|^\sigma} < \infty.
\end{equation}

\subsection{Basic function space setting} Set
\begin{equation}\label{defiH}
   \LO:= L^2_0(\RN) = \big\{ v \in L^2(\RN): v=0~\text{a.e.~in $\RR^N\setminus \Omega$} \big\}.  
\end{equation}
Then $\LO$ is a closed subspace of $\LR$ and is endowed with its standard scalar product,
$$
(u,v) := \int_{\RN} u(x)v(x) \, \dx \quad \mbox{ for } \ u,v \in \LR,
$$
which also induces the norm of $\LO$, i.e., $\|\cdot\|_{\LO}^2 = \|\cdot\|_{\LR}^2 = (\cdot,\cdot)_{\LR}$. Hence $\LO$ is a Hilbert space. Moreover, $L^2(\Omega)$ can be identified with $\LO$ by zero extension outside $\Omega$. Here and henceforth, we simply write $L^2(\Omega)$ instead of $H_0$.

For $s \in (0,1)$, let us recall $\Xsz$ defined by \eqref{Xsz} and furnish $\Xsz$ with the scalar product,
\begin{equation}\label{eq:proscX0}
 (v,z)_{\Xsz} := (v, z) + \frac{C_s}{2} \iint_{\RdN} \dfrac{\left(v(x)-v(y)\right) \left(z(x)-z(y)\right)}{|x-y|^{N+2s}}  \, \dx \, \dy
\end{equation}
for $v,z \in \Xsz$, and with the corresponding norm
\begin{equation}\label{eq:normX0}
 \| v\|^2 := \|v\|^2_{L^2(\Omega)} 
   + \frac{C_s}{2} \iint_{\RdN} \dfrac{|v(x)-v(y)|^2}{|x-y|^{N+2s}} \, \dx \, \dy \quad \mbox{ for } \ v \in \Xsz. 
\end{equation}
Then $\Xsz$ is a Hilbert space (i.e., the above norm is complete) and $\Xsz$ could be also presented in a more familiar form, namely
\begin{equation}\label{eq:defX0_bis}
  \Xsz = \big\{ v\in H^s(\RN) \hbox{ such that } 
     v=0 \hbox{ a.e.~in } \RN\setminus \Omega \big\}.
\end{equation}
Due to the Poincar\'e-type inequality (see Proposition \ref{apdx:P:poincare} in Appendix with $p = 2$),
\begin{equation}\label{eq:poincare}
  \| v\|^2_{L^2(\Omega)}\le c_P \iint_{\RdN} \dfrac{|v(x)-v(y)|^2}{|x-y|^{N+2s}} \, \dx \, \dy \quad \mbox{ for all } \ v\in \Xsz
\end{equation}
for some constant $c_{P}$ depending only on $s$, $N$ and the diameter
of $\Omega$, the norm $\|\cdot\|$ given above is equivalent to
\begin{equation}\label{eq:normX0bis}
  \| v \|^2_{\Xsz} := 
    \frac{C_s}{2} \iint_{\RdN} \dfrac{|v(x)-v(y)|^2}{|x-y|^{N+2s}} \, \dx \, \dy\quad \mbox{ for } \ v \in \Xsz. 
\end{equation}
Therefore, from now on, we will fix \eqref{eq:normX0bis} to be the norm of $\Xsz$. Since $\Xsz$ can be identified with a {\it dense}\/ subspace of $L^2(\Omega)$, one may consider the Hilbert triple,
\begin{equation}\label{H-tr}
\Xsz \hookrightarrow L^2(\Omega) \simeq L^2(\Omega)' \hookrightarrow \Xsz',
\end{equation}
with compact and densely defined canonical injections. This relation will be frequently used throughout the paper. To see this, we recall the following compact embeddings, which might be more or less straightforward; however, a proof will be given below for the convenience of the reader.

\begin{propo}\label{P:embed}
 Suppose that $\Omega$ is a bounded domain of $\mathbb R^N$ with smooth
 boundary.
 For $\sigma > 0$, the space $\Xsigz$ is compactly embedded in
 $L^2(\Omega)$. Moreover, $L^2(\Omega)' (\simeq L^2(\Omega))$ is also compactly embedded in $\Xsigz'$.
\end{propo}

\begin{proof}
 Let $(u_n)$ be a bounded sequence in $\Xsigz$. Then it is also bounded in $W^{\sigma,2}(\Omega)$ from the definition of the norm $\|\cdot\|_{\Xsigz}$ and 
 the Poincar\'e-type inequality \eqref{eq:poincare}. Since $W^{\sigma,2}(\Omega)$ is compactly embedded in $L^2(\Omega)$, $(u_n|_{\Omega})$ is precompact 
 in $L^2(\Omega)$, and hence, $\Xsigz$ is compactly embedded in $L^2(\Omega)$. By Schauder's theorem for compact operators and their adjoint operators, $L^2(\Omega)'$ is also compactly embedded in $\Xsigz'$. 
\end{proof}


\subsection{A variational formulation of fractional Dirichlet Laplacians}

For $s \in (0,1)$, let us define a linear and bounded operator $\As : \Xsz \to \Xsz'$ 
by setting
$$
\left\langle \As u, v \right\rangle_{\Xsz} = \dfrac{C_s}2 \iint_{\RdN} \dfrac{\left(u(x)-u(y)\right)\left(v(x)-v(y)\right)}{|x-y|^{N+2s}} \, \dx \, \dy
$$
for $u,v \in \Xsz$. The right-hand side is a weak form of the fractional Laplacian $\Ds$ (see~\cite{Dine_Pala_Vald}); hence, $\As$ can be regarded as a weak representation of $\Ds$ (see also~\cite{ASSe1}). Here we also remark that $\As$ can be regarded as the Fr\'echet derivative of the convex functional,
$$
Q_s(u) := \dfrac{C_s}4 \iint_{\RdN} \dfrac{|u(x)-u(y)|^2}{|x-y|^{N+2s}} \, \dx \, \dy \quad \mbox{ for } \ u \in \Xsz,
$$
that is, $\As = Q_s'$. Then one can readily find that:
\begin{propo}\label{P:iso}
 $\As$ is an isomorphism from $\Xsz$ to $\Xsz'$.
\end{propo}
\begin{proof}
Note that $Q_s(u) = (1/2)\|u\|_{\Xsz}^2$. By~\cite[Example 2, p.\,53]{B}, $\As = Q_s'$ turns out to be a duality mapping between $\Xsz$ and $\Xsz'$. 
Hence in particular, $\As$ is an isomorphism from $\Xsz$ to $\Xsz'$.
\end{proof}


\subsection{Analytic operator}\label{Ss:analytic}

 We next recall the notion of (\emph{real}) \emph{analyticity} of an operator $T : X \to Y$ defined on Banach spaces $X,Y$ (see, e.g., Definition 8.8 of~\cite{ZeidlerI}).
\begin{defi}
Let $X, Y$ be Banach spaces and let $T : X \to Y$ be an operator.
Then $T : X \to Y$ is said to be \emph{analytic} at $z \in X$ if there exist $r > 0$ and
for each $n \in \mathbb N$ also exists a symmetric bounded $n$-linear operator $T_n(z) : X^n \to Y$  such that
\begin{align}\label{analy-1}
T(z+h) - T(z)
&= \sum_{n=1}^\infty [T_n(z)](\underbrace{h,\ldots,h}_{\text{$n$ times}}) \ \mbox{ in } Y
\end{align}
for any $h \in X$ satisfying $\|h\|_X < r$, and
\begin{equation}\label{analy-2}
 \sup_{n \in \N} \|T_n(z)\|_{\mathcal L^n(X,Y)} r^n < \infty.
\end{equation}
Let $U$ be an open set in $X$. If $T$ is analytic at each $z \in U$, then $T$ is said to be analytic in $U$. 
\end{defi}
\begin{remark}
 {\rm
 Under \eqref{analy-2}, one can check that
 $$
 \sum_{n = 1}^\infty \|T_n(z)\|_{\mathcal L^n(X,Y)} \|h\|_X^n < \infty
 \quad \mbox{ if } \ h \in X, \quad \|h\|_X < r.
 $$
 }
\end{remark}


\subsection{\L ojasiewicz inequality}

Let us finally recall a classical \L ojasiewicz gradient inequality for analytic functions defined on finite dimensional spaces.
\begin{propo}[{\L}ojasiewicz~\cite{L,L2}]\label{P:L}
 Let $x_0 \in \RN$ and let $f$ be a real analytic function defined on a neighbourhood $U$ of $x_0$. Then there exist constants $\theta \in (0,1/2]$ and $C,\delta > 0$ such that
 \begin{equation}\label{Lo}
  |f(x_0)-f(x)|^{1-\theta} \leq C|\nabla f(x)|
 \end{equation}
 for all $x \in U$ satisfying $|x-x_0| < \delta$.
\end{propo}


\section{Main results}
\label{sec:main}

This section is devoted to stating the main results of the present paper. We assume $g$ to satisfy at least the following basic property:
\begin{equation}\label{hp:g1}
  g \in C^1(\RR), \quad g(0) = 0.
\end{equation}
Letting $\gciapo\in C^2(\RR)$ denote a primitive function of $g$, we define the {\it energy functional}\/ $\J$ as follows:
\begin{equation}\label{defiJ}
  \J(v) := \dfrac{C_\sigma}4 \iint_{\mathbb R^{2N}}
   \dfrac{|v(x)-v(y)|^2}{|x-y|^{N+2\sigma}} \,\dix\,\diy
   + \int_{\Omega} \gciapo(v(x)) \,\dix
\end{equation}
for $v$ satisfying
\begin{equation}\label{gv-}
  v\in \Xsigz, \quad
   \gciapo(v)^- \in L^1(\Omega),
\end{equation}
where $(\cdot)^-$ denotes the negative part function.

Throughout this paper, we are concerned with solutions to the fractional Cahn-Hilliard system \eqref{eq:fCH}--\eqref{eq:bc} defined by

\begin{defi}[Weak solutions]\label{def:weak}
 Let $0\le S < T \le \infty$. Let $g$ satisfy\/ \eqref{hp:g1}.
 We say that $(u,w)$ is a\/ {\rm weak solution} to the fractional Cahn-Hilliard
 system \eqref{eq:fCH}--\eqref{eq:bc} on the interval $(S,T)$ if 
 \begin{align}\label{reg:u}
   & u \in C_w([S,T);\Xsigz), \quad
    u_t \in L^2(S,T;\Xsz'),\\ 
  \label{reg:w}
  & w \in L^2(S,T;\Xsz),\\
  & g(u) \in L^2_{\loc}(S,T;L^2(\Omega)), \label{reg:gu}
 \end{align}
 and the following equations hold a.e.~in $(S,T)$\/{\rm :}
 \begin{align}
  \langle u_t(t), z \rangle_{\Xsz} + \dfrac{C_s}2 \iint_{\RdN} \dfrac{\left(w(x,t)-w(y,t)\right)\left(z(x)-z(y)\right)}{|x-y|^{N+2s}} \,\dx\,\dy = 0 \nonumber\\
  \quad \mbox{ for all } \ z \in \Xsz \ \mbox{ and a.e. } \ t \in (S,T) \label{eq:w0}
 \end{align}
 and
  \begin{align}
   \int_\Omega w(x,t) \zeta(x) \, \dx = \dfrac{C_\sigma}2 \iint_{\RdN} \dfrac{\left(u(x,t)-u(y,t)\right)\left(\zeta(x)-\zeta(y)\right)}{|x-y|^{N+2\sigma}} \,\dx\,\dy \nonumber\\
+\int_\Omega g(u(x,t))\zeta(x) \, \d x  \quad \mbox{ for all } \ \zeta \in \Xsigz \ \mbox{ and a.e. } \ t \in (S,T).\label{eq:u0}
 \end{align}
\end{defi}
Prior to exhibiting basic assumptions, we give the following
 \begin{remark}\label{rem:reg}{\rm
  \begin{enumerate}
   \item[(i)] Weak forms \eqref{eq:w0} and \eqref{eq:u0} can be also equivalently rewritten as
\begin{align}\label{eq:w}
   & u_t + \As w = 0 \ \mbox{ in } \Xsz',\\
  \label{eq:u}
   & w = \Asig u + g(u) \ \mbox{ in } \Xsigz'.
\end{align}
   \item[(ii)] By \eqref{hp:g1} and \eqref{reg:u},
 $g(u)$ is measurable and vanishes identically outside $\Omega$. Moreover,
 the regularity
 \begin{equation}\label{reg:gu0}
   g(u) \in L^2_{\loc}(S,T;\Xsigz'),
 \end{equation}
 is implicitly hidden in equation \eqref{eq:u}. Actually, if \eqref{eq:u} holds, then \eqref{reg:gu0} follows from
 a comparison of terms thanks to \eqref{reg:u}-\eqref{reg:w}.
 However, \eqref{reg:gu} does not directly follow from the definition of weak solutions mentioned above. 
 It is worth observing that, in our existence theorem we shall
 need stronger assumptions on $g$ (see below), and correspondingly,
 we shall get better regularity for $g(u)$.
  \end{enumerate}
 }
\end{remark}
%
%
%

\noindent%

In order to ensure existence of weak solutions, we need to assume, beyond \eqref{hp:g1}, a couple of additional conditions, which will be referred to as {\sl $\lambda$-monotonicity}, and {\sl dissipativity}, respectively:
\begin{align}\label{hp:g3}
  & \text{There exists } \ \lambda \ge 0 
    \ \text{ such that } \ g'(r) \ge -\lambda 
    \ \text{ for all } r \in \RR,\\
 \label{hp:g4}
  & \liminf_{|r|\to \infty}
   \big( g(r)r + (\lambda_1 - \kappa) r^2 \big) > 0
   \ \text{ for some }\ \kappa>0,
\end{align}
where $\lambda_1=\lambda_1(\sigma)>0$ is the first eigenvalue of $\Dsig$ (see~\cite{serva-valdi11}). Namely,
$$
\lambda_1 = \inf_{v \neq 0} \dfrac{\|v\|_{\Xsigz}^2}{\|v\|_{L^2(\Omega)}^2} > 0.
$$
Note that, if \eqref{hp:g3} holds, then,
setting $\beta(r):= g(r) + \lambda r$ for $r\in \RR$, we find by \eqref{hp:g1} that $\beta$ is of class $C^1$ and monotone and that $\beta(0) = 0$. Moreover, \eqref{eq:u} can be 
equivalently rewritten as 
\begin{equation}\label{eq:ubeta}
   w = \Asig u + \beta(u) - \lambda u \ \mbox{ in } \Xsigz'.
\end{equation}
Observe also that, if \eqref{hp:g4} holds, then we can easily prove that
\begin{equation}\label{g4-11}
  \io \gciapo(u) \, \dx \ge -\frac{\lambda_1-\kappa}2 \| u \|_{L^2(\Omega)}^2 - C,
\end{equation}
for some $C\ge 0$, whence (cf., e.g., \cite{ASSe1})
the energy satisfies the basic coercivity property,
\begin{equation}\label{g4-12}
  \J(v) \ge \kappa_0 \| v \|_{\Xsigz}^2 - C 
   \quext{ for all }\ v \in \Xsigz,
\end{equation}
%
where $\kappa_0:=\kappa/(2\lambda_1)>0$.
%
%

Let us also specify some natural assumptions on the initial datum:
\begin{equation}\label{hp:u0}
  u_0 \in \Xsigz, \quad \betaciapo(u_0)\in L^1(\Omega),
\end{equation}
where $\betaciapo$ stands for a primitive function of $\beta$ (i.e., $\betaciapo' = \beta$).
We remark that \eqref{gv-} follows immediately from \eqref{hp:u0}, and moreover, \eqref{hp:u0}
exactly corresponds to the finiteness of the initial energy, namely we have $\J(u_0)<\infty$.

The first result of this paper concerns existence of global weak solutions.
The proof will be only sketched since it essentially consists of a small modification of the 
argument given in~\cite{ASSe1}.
\begin{theorem}[Existence and uniqueness of weak solutions]\label{thm:exi}
 Let us assume\/ \eqref{hp:g1}, \eqref{hp:g3} and\/ \eqref{hp:g4},
 and let $u_0$ satisfy\/ \eqref{hp:u0}. Then, there exists one and only 
 one weak solution $(u,w)$ of \eqref{eq:fCH}--\eqref{eq:bc} in the sense of Definition \ref{def:weak} defined over $(0,\infty)$. Moreover, the function $t \mapsto \J(u(t))$ is non-increasing and right-continuous in $[0,\infty)$ and differentiable a.e.~in $(0,\infty)$. The function $t \mapsto u(t)$ is also right-continuous on $[0,\infty)$ in the strong topology of $\Xsigz$. Furthermore, the following\/ {\rm energy inequalities} hold\/{\rm :}
 \begin{align}\label{ei1}
   & \|w(t)\|_{\Xsz}^2 + \ddt \J(u(t)) \le 0
    \quad \mbox{ for a.a. } t \in (0,\infty),\\
  \label{ei2}
   & \langle u_t(t), \Asig u(t) + g(u(t)) \rangle_{\Xsz} 
    \ge \ddt \J(u(t)) \quad \mbox{ for a.a. } t \in (0,\infty),\\
  \label{ei-i}  
   & \int^t_\tau \|w(r)\|_{\Xsz}^2 \,\dir
    + \J(u(t)) \le \J(u(\tau))
    \quad \mbox{ for all } \ 0 \leq \tau \leq t.
 \end{align}
 Finally, for any $T > 0$, there exists a constant $C_T > 0$ such that
 \begin{equation}
  \sup_{t \geq 0} \int^{t+T}_t \|\beta(u(\tau))\|_{L^2(\Omega)}^2 \, \d \tau \leq C_T.\label{reg:beta}
 \end{equation}
\end{theorem}

Under the same conditions on $g$, we also find out parabolic smoothing
properties of weak solutions:
\begin{theorem}[Smoothing effect]\label{thm:reg}
 Let the assumptions of\/ {\rm Theorem~\ref{thm:exi}} hold. Then, for any $t_0>0$,
 we additionally have 
\begin{align}
  & u_t \in L^2(t_0,\infty;\Xsigz) \cap L^\infty(t_0,\infty;\Xsz'),\label{reg:u2}\\ 
  \label{reg:w2}
   & w \in L^\infty(t_0,\infty;\Xsz).
 \end{align}
\end{theorem}

In some cases, we can also derive energy {\sl equalities}.
\begin{theorem}[Energy equalities]\label{thm:ener}
 Let the assumptions of\/ {\rm Theorem~\ref{thm:exi}} hold and
 let $(u,w)$ be a weak solution defined over the interval $(S,T)$ 
 with $0\le S < T \le \infty$ and additionally satisfying either
 \begin{equation}\label{reg:u3}
   u \in L^2_{\loc}(S,T; \Xsz )
 \end{equation}
 or
 \begin{equation}\label{reg:u4}
   u_t \in L^2_{\loc}(S,T;L^2(\Omega)).
 \end{equation}
 Then, relations\/ \eqref{ei1}-\eqref{ei-i} hold with inequalities replaced by the equal sign over the interval $(S,T)$.
\end{theorem}
\begin{remark}\label{rem:cond}{\rm 
 The above is in fact a conditional result, in the sense 
 that \eqref{reg:u3} and \eqref{reg:u4} are hypotheses. 
 In the sequel we shall provide a number of actual
 situations where the above assumptions are satisfied.
 In particular this happens when
 $\sigma \ge s$ (so that~\eqref{reg:u3} follows 
 from~\eqref{reg:u}) and under the conditions of Theorem~\ref{thm:reg}
 (when \eqref{reg:u4} follows from \eqref{reg:u2} at least
 for $S>0$). 
 }
\end{remark}

The next proposition is concerned with the (eventual) boundedness of $u = u(x,t)$:

\begin{propo}[Boundedness of $u(x,t)$]\label{P:hoelder}
 Let the assumptions of\/ {\rm Theorem~\ref{thm:exi}} hold and let $(u,w)$ be a weak solution of \eqref{eq:fCH}--\eqref{eq:bc} on $(0,\infty)$. Then, for any $t\ge 1$, it holds that
 $$
 \Asig u(t) + \beta(u(t)) = \lambda u(t) + w(t) \in L^p(\Omega)
 $$
 for $p=s^*:=\frac{2N}{N-2s}$ if $N>2s$ and for any 
 $p\in [1,\infty)$ if $N\le 2s$.
 Moreover, if
 \begin{equation}\label{big:ssig}
   2s + 4 \sigma \BBB > \SSS N,
 \end{equation}
 then there exists a constant $\alpha>0$ depending on $N,s,\sigma$ such that
 \begin{equation}\label{u:bound}
   \| u(t) \|_{C^\alpha(\barO)} \le C
    \quad \mbox{ for all } \ t\ge 1.
 \end{equation}
\end{propo}
%

We are ready to investigate the long-time behavior of solution trajectories. 
Let us first discuss existence of nonempty $\omega$-limit sets.
\begin{lemma}[Nonempty $\omega$-limit set]\label{L:omega}
 Let the assumptions of\/ {\rm Theorem \ref{thm:exi}} hold and let $(u,w)$ be the unique weak solution of \eqref{eq:fCH}--\eqref{eq:bc} on $(0,\infty)$ as provided by the theorem. Then, for any sequence $t_n \to \infty$, one can take a (not relabeled) subsequence of $(t_n)$ and $\phi \in \Xsigz$ such that 
 \begin{equation}\label{subs}
   u(t_{n}) \to \phi \ \mbox{ strongly in } \Xsigz \quad \mbox{ and } \quad \Esi(u(t_n)) \to \Esi(\phi),
 \end{equation}
 and moreover, $\phi$ solves the stationary problem,
\begin{equation}\label{elliptic-f}
  g(\phi) \in L^2(\Omega) \quad \mbox{and} \quad \Asig \phi + g(\phi) = 0 \ \mbox{ in } \Xsigz'.
\end{equation}
 In particular, the $\omega$-limit set of $u$ is nonempty 
 and it is contained into the set of all solutions to \eqref{elliptic-f}.
\end{lemma}

Let us now show that, under additional assumptions, the $\omega$-limit
set of any weak solution is a singleton. This will be proved by using
a variant of the so-called \L ojasiewicz-Simon inequality suitable for
fractional Dirichlet Laplace operators.
To this end, we introduce notions of \emph{real analyticity} of $g$ as follows:
 \begin{enumerate}
 \item[(H1)] {\bf (Uniform analyticity)} Let $0 < a, b \leq \infty$. Assume that $g \in C^\infty(-a,b)$, and moreover, there exist constants $C,M \geq 0$ such that, for all $s \in (-a,b)$ and $n \in \N$ large enough,
	     $$
	     |g^{(n)}(s)| \leq C M^n n!.
	     $$
 \item[(H2)] {\bf (Analyticity with a singularity at the origin)} Let $0 < b \leq \infty$. Assume that $g \in C^\infty(0,b)$, and moreover, there exist constants $C,M \geq 0$ such that, for all $s \in (0,b)$ and $n \in \mathbb N$ large enough,
	     $$
	     |g^{(n)}(s)| \leq C \dfrac{M^n n!}{|s|^n}.
	     $$
 \end{enumerate}
When either $a$ or $b$ is infinite, we further assume the so-called Sobolev subcritical growth condition,
\begin{enumerate}
 \item[(H3)] There exist constants $C \geq 0$ and $0 \leq p \leq \frac{N+2\sigma}{(N-2\sigma)_+}$ such that
	    \begin{equation}\label{hypo-sig1}
	     |g'(s)| \leq C(|s|^{p-1}+1) \quad \mbox{ for all } \ s \in \R.
	    \end{equation}
\end{enumerate}
\begin{remark}\label{R:analytic}
 {\rm
 \begin{enumerate}
  \item[(i)] In case (H1) is satisfied, by Taylor's theorem, $g(s)$ can be uniformly expanded as follows:
	     \begin{equation}\label{taylor}
	      g(s) = \sum_{n=0}^\infty \dfrac{g^{(n)}(s_0)}{n!}(s-s_0)^n
	     \end{equation}
	     converges uniformly for $s_0 \in (-a,b)$ and $s \in (-a,b) \cap (s_0-(2M)^{-1},s_0+(2M)^{-1})$. Typical examples of $g(s)$ satisfying (H1) would be polynomial and trigonometric functions (with $a = -\infty$ and $b = \infty$) and exponential and hyperbolic functions (with finite $a$, $b$). In case (H2) is satisfied, one cannot ensure the uniform convergence of \eqref{taylor} in $(0,\vep)$, for $g^{(n)}$ may have a singularity at the origin. A typical example of the case would be power functions $g(s) = s^m$ with noninteger $m \geq 0$. In view of \eqref{hp:g1}, $m$ is restricted to be not less than $1$ (then $g \in C^1$), and hence, the case $m < 1$ is beyond the scope.
\item[(ii)] In particular, (H2) is equivalent to the condition that there exists $\theta \in (0,\pi/2)$ such that $g$ can be extended as a (complex) analytic function on the sector $S_{\theta,\beta} = \{ z \in \mathbb C \colon |z| \in (0,\beta), \ \mathrm{Arg}\, z \in (-\theta,\theta)\}$ (in particular, $g$ is real analytic in $(0, \beta)$).
  \item[(iii)] (H3) implies that there exists a constant $C \geq 0$ such that
	       \begin{align}
		|g(s)| &\leq C \left(|s|^{p}+1\right) \quad \mbox{ for all } \ s \in \R,\label{g-ciappo}\\
		|\widehat{g}(s)| &\leq C \left(|s|^{p+1}+1\right) \quad \mbox{ for all } \ s \in \R.\label{g}
	       \end{align}
	       Hence the functional
	       $$
	       G(u) := \int_\Omega \widehat{g}(u(x)) \, \d x \quad \mbox{ for } \ u \in \Xsigz
	       $$
	       turns out to be of class $C^2$ in $\Xsigz$, since $g \in C^1(\R)$ by \eqref{hp:g1} and $\Xsigz \hookrightarrow L^{p+1}(\Omega)$ by $1 \leq p+1 \leq 2N/(N-2\sigma)_+$. In particular, $G':u \mapsto g(u(\cdot))$ is a Nemytskii operator of class $C^1$ from $\Xsigz$ to $\Xsigz'$. 
 \item[(iv)] Throughout this paper, we always focus on the behavior of $g(u)$ around the origin $u = 0$, since the homogeneous Dirichlet boundary condition is imposed and solutions $u(x,t)$ and equilibria $\phi(x)$ take values around zero. Therefore we treat the cases (H1) and (H2) only. Namely, $g(u)$ is either uniformly analytic in an open interval including $0$ or analytic in $(0,\vep)$ with a singularity at the origin. However, one can also generalize the results of the present paper, in particular, \L S inequality (see Theorem \ref{T:LSI-bdd} below), to the case where $g(u)$ is analytic in an open interval $I$ (and $g(u)$ may have singularity on the boundary of $I$) in an analogous way.
 \end{enumerate}
 }
\end{remark}
Our main result reads,
\begin{theorem}[Convergence of solutions to equilibria]\label{thm:Loj}
 Let \eqref{hp:g1}, \eqref{hp:g3} and \eqref{hp:g4} hold and let $(u,w)$ be a weak solution of \eqref{eq:fCH}--\eqref{eq:bc} defined over $(0,\infty)$. 
Let $\phi \in \Xsigz$ be a solution to \eqref{elliptic-f} satisfying \eqref{subs} for some sequence $t_n \to \infty$. In addition, assume one of the following {\rm (i)--(iv)}\/{\rm :}
\begin{enumerate}
 \item[\rm (i)] Assume that {\rm (H1)} and {\rm (H3)} hold with $a = b = +\infty$. 
 \item[\rm (ii)] Assume that {\rm (H1)} holds with some $a, b \in (0,\infty)$ and that
	      $$
	      \|\phi\|_{L^\infty(\Omega)}, \|u\|_{L^\infty(\Omega \times (\tau,\infty))} < a \wedge b
	      $$
	      for some $\tau > 0$.
 \item[\rm (iii)] Assume that {\rm (H2)} and {\rm (H3)} hold with $b = \infty$ and that $\phi > 0$ a.e.~in $\Omega$. 
 \item[\rm (iv)] Assume that {\rm (H2)} holds with some $b \in (0,\infty)$ and that
	      $$
	      0 < \phi < b \ \mbox{ a.e.~in } \Omega, \quad \|u\|_{L^\infty(\Omega \times (\tau,\infty))} < b
	      $$
	      for some $\tau > 0$.
\end{enumerate}
 Then the whole trajectory $\{u(t) \colon t \geq 0\}$ converges to the stationary solution, namely
\begin{equation}\label{alltraj}
   u(t) \to \phi \quad \mbox{ strongly in } \Xsigz \ \mbox{ as } \ t \to +\infty.
\end{equation}
\end{theorem}

\begin{remark}\label{rem:bound}{\rm
If $g$ satisfies the sign condition,
\begin{equation}\label{g:sign}
  g(r)\sign r > 0 
   \quext{for all } \ |r| > \gamma
\end{equation}
 for some $\gamma > 0$, it then follows that
\begin{equation}\label{hypo-phi}
 \mathrm{ess}\,\sup_{x \in \Omega} |\phi(x)| < \infty
\end{equation}
for any equilibria (i.e., solutions to \eqref{elliptic-f}) $\phi$. 
Indeed, \eqref{hypo-phi} can be immediately proved
by elementary maximum principle arguments. Namely, one 
may test the equation in \eqref{elliptic-f} by $(\phi - \gamma)^+$ and
by $- (\phi + \gamma)^-$, where $(s)^\pm := \max \{\pm s, 0\} \geq 0$ for $s \in \R$.
}
\end{remark}

 Combining Proposition \ref{P:hoelder} and Theorem \ref{thm:Loj}, we readily obtain
 \begin{corollary}
  Let $u = u(x,t)$ be a solution of \eqref{eq:fCH}--\eqref{eq:bc} on $(0,\infty)$ such that the $\omega$-limit set $\omega(u)$ of $u$ contains an equilibrium $\phi$. 
  Taking the assumptions of Theorem \ref{thm:exi}, together with
  \eqref{big:ssig} and \eqref{g:sign}, it holds that $\phi$ belongs to $L^\infty(\Omega)$ and $\|u(t)\|_{C^\alpha(\overline\Omega)} \leq C$ for all $t \geq 1$.
  Hence if either {\rm (H1)} or {\rm [}{\rm (H2)} along with $\phi > 0$ in $\Omega${\rm ]} holds for some $a,b > 0$, then the $\omega$-limit set of $u$ contains the equilibrium $\phi$ only. 
 \end{corollary}

The following theorem will play a key role to prove Theorem \ref{thm:Loj}.

 \begin{theorem}[\L S inequality for fractional Dirichlet  Laplacian]\label{T:LSI-bdd}
Assume \eqref{hp:g1}. Let $\sigma \in (0,1)$ and let $\phi \in \Xsigz \cap L^\infty(\Omega)$ be a solution of the stationary problem \eqref{elliptic-f}.
\begin{enumerate}
 \item[\rm (a)] Suppose that either {\rm (\ichi)} or {\rm (\san)} holds\/{\rm :}
 \begin{enumerate}
  \item[\rm (\ichi)] {\rm (H1)} and {\rm (H3)} are satisfied with $a = b = \infty$. 
  \item[\rm (\san)] {\rm (H2)} and {\rm (H3)} hold with $b = \infty$ and $\phi > 0$ a.e.~in $\Omega$. 
 \end{enumerate}
 Then there exist $\theta \in (0,1/2]$ and $\omega, \delta > 0$ such that
\begin{equation}\label{LSI-2}
\left| \Esi(v) - \Esi(\phi) \right|^{1-\theta}
\leq \omega \left\| \Asig v + g(v) \right\|_{\Xsigz'},
\end{equation} 
	    whenever $v \in \Xsigz$ and $\|v-\phi\|_{\Xsigz} < \delta$.
 \item[\rm (b)] Let $\eta > 0$ and suppose that either {\rm (\nii)} or {\rm (\shi)} holds\/{\rm :}
 \begin{enumerate}
  \item[\rm (\nii)] {\rm (H1)} and $\|\phi\|_{L^\infty} < \gamma$ are satisfied    with $a,b,\gamma > 0$ satisfying $\gamma, \eta < a \wedge b < \infty$.
  \item[\rm (\shi)] {\rm (H2)} holds and $0 < \phi < \gamma$ a.e.~in $\Omega$ with $b,\gamma> 0$ satisfying $\gamma, \eta < b < \infty$.
 \end{enumerate}
 Then there exist $\theta \in (0,1/2]$ and $\omega, \delta > 0$ such that
\begin{equation}\label{LSI-2+}
\left| \Esi(v) - \Esi(\phi) \right|^{1-\theta}
\leq \omega \left\| \Asig v + g(v) \right\|_{\Xsigz'},
\end{equation} 
 whenever $v \in \Xsigz$, $\mathrm{ess\,sup}_{x \in \Omega}|v(x)| < \eta$ and $\|v-\phi\|_{\Xsigz} < \delta$.
\end{enumerate}
 \end{theorem}

 \begin{remark}
  {\rm
\begin{enumerate}
 \item[(i)] One can also treat the case where $g$ is analytic on $(-a,0)$ with a singularity at the origin (cf.~(H2)) and $-a < \phi < 0$ a.e.~in $\Omega$ by 
  performing the transform $u \mapsto -u$, $\phi \mapsto -\phi$ and $g(\cdot) \mapsto g(-\ \cdot\,)$ and applying Theorem \ref{T:LSI-bdd}. Moreover, by translation, 
  one may further generalize the inequality to $g(\cdot)$ analytic on more general intervals $I$ (which may not include zero and may have singularity on the boundary) and $\phi(x) \in I \setminus \partial I$ a.e.~in $\Omega$.
 \item[(ii)] When $\phi$ is a regular point of $\J$ (i.e., $\J'(\phi)\neq 0$), 
 inequalities \eqref{LSI-2}, \eqref{LSI-2+} follow immediately from the $C^1$ regularity 
  of $\J$ in $\Xsigz$. So \eqref{LSI-2}, \eqref{LSI-2+} also hold true for any $\phi \in \Xsigz$.
 \item[(iii)] Assertion (a) of Theorem~\ref{T:LSI-bdd} for the case (i)
  can be also proved by using the abstract theory developed in~\cite{Chill03}. 
\end{enumerate}
  }
 \end{remark}

 For the classical (Dirichlet) Laplace operator $\Delta$ (i.e., the case $\sigma = 1$) in $L^p(\Omega)$-spaces ($1 < p < \infty$), the domain of $\Delta$ coincides with $W^{2,p}(\Omega) \cap H^1_0(\Omega)$. Indeed, according to the Calderon-Zygmund singular integral theory, $u$ belongs to $W^{2,p}(\Omega) \cap H^1_0(\Omega)$, provided that $\Delta u \in L^p(\Omega)$. However, it is worth mentioning that, for general $\sigma \in (0,1)$, a corresponding property may not be true. To be more precise, even if $\Dsig u$ belongs to $L^p(\Omega)$, it may be false that $u \in W^{2\sigma,p}(\Omega)$. The domain of $\Dsig$ is still unclear in the $L^p(\Omega)$ framework. Furthermore, in contrast with the Schauder theory, $u \in C^\sigma(\overline\Omega)$ at most, even though $\Dsig u \in C^\infty(\overline\Omega)$. For more details, we refer the reader to~\cite[Remarks 7.1 and 7.2]{RoSe13} and~\cite{RoSe12}. This fact prevents us to directly apply proofs of \L S inequalities for the classical Laplacian. Indeed, they are based on $W^{m,p}(\Omega)$ or $C^{m}(\overline\Omega)$ frameworks, where a linearized operator is defined (see, e.g.,~\cite{Simon83,FeiSim00,RyHo}).
 To overcome such a difficulty, for $p \in (1,\infty)$, we introduce the space
  $$
\Xp := \left\{
u \in \Xsigz \cap L^p(\Omega) \colon \Asig u \in L^p(\Omega)
\right\}.
$$
 This acts as the natural domain of the $\Dsig$ seen as an unbounded linear operator of $L^p(\Omega)$.
 We cannot characterize the elements of $\Xp$ in terms of regularity. On the other hand, as shown below,
 if $\Xp$ is equipped with the graph norm
$$
\|u\|_{\Xp} := \|u\|_{L^p(\Omega)} + \|u\|_{\Xsigz} + \|\Asig u\|_{L^p(\Omega)} \quad \mbox{ for } \ u \in \Xp,
$$
then it gains good properties and can be used as a space for the long-time analysis. 

The following proposition will play a crucial role to prove not only Theorem \ref{T:LSI-bdd} but also Proposition \ref{P:hoelder}.
\begin{propo}\label{new:P0}
 The following {\rm (i)--(iii)} hold true\/{\rm :}
 \begin{enumerate}
  \item[\rm (i)] $\Xp$ is a uniformly convex Banach space.
  \item[\rm (ii)] If $\frac N {2\sigma} < p < \infty$, then $\Xp$ is
	       continuously embedded in $C^\beta({\overline\Omega})$ with $\beta = \sigma \wedge (2\sigma - \frac N p )$.
  \item[\rm (iii)] $|u|_{\Xp} := \|\Asig u\|_{L^p(\Omega)}$ is also an
	equivalent norm to $\|\cdot\|_{\Xp}$, provided that $p \geq \frac{2N}{N+2\sigma}$.
 \end{enumerate}
\end{propo}

 \begin{proof}
We first prove (i). One can easily check that $\|\cdot\|_{\Xp}$ is a norm of $\Xp$. So let us next check that $\Xp$ is complete. Let $(u_n)$ be a Cauchy sequence in $\Xp$. Then $u_n$ converges to $u$ strongly in $\Xsigz \cap L^p(\Omega)$, and hence, $\Asig u_n \to \Asig u$ strongly in $\Xsigz'$. Moreover, since $(\Asig u_n)$ forms a Cauchy sequence in $L^p(\Omega)$, we find that $\Asig u_n \to \Asig u$ strongly in $L^p(\Omega)$. Thus $u \in \Xp$ and $u_n \to u$ strongly in $\Xp$. Moreover, the uniform convexity readily follows from the definition of $\|\cdot\|_{\Xp}$.

As for (ii), due to~\cite[Proposition 1.4]{RoSe13}, if $\frac N {2\sigma} < p < \infty$, then we see that
$$
\|u\|_{C^\beta(\mathbb R^N)} \lesssim \|\Asig u\|_{L^p(\Omega)}
\quad \mbox{ for } \ u \in \Xp, \quad \mbox{ with } \ \beta = \sigma \wedge (2\sigma - N/p),
$$
which implies (ii).
  
Recalling~\cite[Proposition 1.4]{RoSe13} again, we deduce that, for any $1 \leq p < \infty$,
$$
\|u\|_{L^p(\Omega)} \lesssim \|\Asig u\|_{L^p(\Omega)}
\quad \mbox{ for } \ u \in \Xp.
$$
Moreover, if $p \geq \frac{2N}{N+2\sigma}$ (equivalently, $p' \leq
 \frac{2N}{N-2\sigma}$), then one has
$$
\|u\|_{L^{p'}(\Omega)} \lesssim \|u\|_{\Xsigz}
\quad \mbox{ for } \ u \in \Xsigz.
$$
Hence it holds that, for $u \in \Xp$, i.e., $\Asig u \in L^p(\Omega)$,
\begin{align*}
 \|u\|_{\Xsigz}^2
&= \dfrac{C_\sigma}2 [u]_{H^\sigma(\RN)}^2\\
&= \left\langle \Asig u, u
 \right\rangle_{\Xsigz}
\leq \|\Asig u\|_{L^p(\Omega)} \|u\|_{L^{p'}(\Omega)}
\lesssim \|\Asig u\|_{L^p(\Omega)} \|u\|_{\Xsigz},
\end{align*}
whence follows
$$
\|u\|_{\Xsigz} \lesssim \|\Asig u\|_{L^p(\Omega)}
\quad \mbox{ for } \ u \in \Xp.
$$
Therefore $|\cdot|_{\Xp}$ is equivalent to $\|\cdot\|_{\Xp}$, and thus, (iii) is proved.
 \end{proof}
Here we also remark that
\begin{remark}
 {\rm
 For any $h \in L^r(\Omega) \subset \Xsigz'$ with $r \gg 1$, the unique weak solution $u
 \in \Xsigz$ of
\begin{equation}\label{elliptic-w}
\Asig u = h \ \mbox{ in } \Xsigz'
\end{equation}
exists. Here we further note that $u$ is also a solution of the
 Dirichlet problem
\begin{equation}\label{elliptic}
\Dsig u = h \ \mbox{ in } \Omega, \quad u = 0 \ \mbox{ in } \mathbb R^N
\setminus \Omega
\end{equation}
in the sense of~\cite{RoSe12, RoSe13}. Indeed, since the weak solution
$u$ belongs to $\Xsigz \hookrightarrow H^\sigma(\mathbb R^N)$, we find that
$\Dsigm u \in L^2(\mathbb R^N)$. Therefore, by the Plancherel theorem,
 it follows that
\begin{align*}
\int_\Omega h \varphi ~\dx
&= \left\langle \Asig u, \varphi \right\rangle_{\Xsigz}\\
&= \dfrac{C_\sigma}2\iint_{\mathbb R^{2N}} 
\dfrac{\left(u(x)-u(y)\right)\left(\varphi(x)-\varphi(y)\right)}{|x-y|^{N+2\sigma}}
 ~\dx \, \dy\\
&= \int_{\mathbb R^N_\xi} |\xi|^{2\sigma} \widehat u(\xi) \widehat
 \varphi(\xi) ~\d \xi\\
&= \int_{\mathbb R^N} 
\Dsigm u \; \Dsigm \varphi
~\dx
\quad \mbox{ for any } \ \varphi \in \Xsigz,
\end{align*}
which is nothing but the definition of solution of \eqref{elliptic}
 in~\cite{RoSe12,RoSe13}.
So one can apply the results of~\cite{RoSe12,RoSe13} to weak solutions
 of \eqref{elliptic-w} as well.
}
\end{remark}

\section{Existence and regularization of weak solutions}
\label{sec:para}

In this Section we give highlights of proofs of Theorems~\ref{thm:exi}--\ref{thm:ener} and Proposition \ref{P:hoelder}.


\subsection{Proof of Theorem~\ref{thm:exi}}\label{ss:ex}%

We first observe that all assertions of Theorem \ref{thm:exi} except \eqref{reg:gu} (in Definition \ref{def:weak}) and \eqref{reg:beta} 
can be proved as in~\cite{ASSe1}, where \eqref{reg:gu} is actually proved for a power nonlinearity 
$g(u) = |u|^{p-2}u - \lambda u$ with $p \in (1,\infty) \setminus \{2\}$.
So it remains to show \eqref{reg:gu} and \eqref{reg:beta} for general $g(u)$ satisfying \eqref{hp:g1}, \eqref{hp:g3} and \eqref{hp:g4}. 
To this aim, we first approximate $\beta$ by its Yosida approximation $\beta_\vep$ for $\vep  > 0$. Then $\beta_\vep$ is a linearly growing maximal
monotone function of class $C^1$ (due to \eqref{hp:g1} and definition of Yosida approximation). Then one can verify that 
$g_\vep(s) := \beta_\vep(s) - \lambda s$ also fulfills \eqref{hp:g1} and \eqref{hp:g3} (indeed, \eqref{hp:g4} is not necessary 
to construct a solution on an arbitrary finite interval $[0,T]$). For any $T > 0$ and each $\vep > 0$, one can construct a 
solution $(u_\vep,w_\vep)$ on $[0,T]$ of \eqref{eq:fCH}--\eqref{eq:bc} with $g$ replaced by $g_\vep$ and derive corresponding 
energy inequalities \eqref{ei1}--\eqref{ei-i} as in~\cite{ASSe1}, where the power function $\beta(s) = |s|^{q-2}s$ is treated 
and whose existence result can be easily extended to smooth nonlinearities with power growth. Moreover, as in~\cite{ASSe1}, 
one tests a (regularized) equation by $\beta_\vep(u_\vep)$ to get
$$
\| \beta_\vep(u_\vep(t)) \|_{L^2(\Omega))}^2 \leq C \| w_\vep(t) \|_{L^2(\Omega)}^2
   + C \| u_\vep(t) \|_{L^2(\Omega)}^2
   \quad \mbox{ for a.e. } t \in (0,T),
$$
where $C$ is independent of $\vep$.
Thus $\beta_\vep(u_\vep)$ turns out to be uniformly bounded in $L^2(0,T;L^2(\Omega))$ with respect to $\vep$
in view of the fact that the right hand side above is uniformly controlled due to the a-priori estimates
resulting from the energy inequality. 
%

Therefore, as in~\cite{ASSe1}, one can pass the limit as $\vep \to 0$ and obtain a solution $(u,w)$ on $[0,T]$ of \eqref{eq:fCH}--\eqref{eq:bc} 
with energy inequalities \eqref{ei1}--\eqref{ei-i} such that $\beta(u) \in L^2(0,T;L^2(\Omega))$. 
In particular, we have, by \eqref{hp:g4} (and hence \eqref{g4-12}),
\begin{equation}\label{exi:11}
  \int_0^t \| w(r) \|_{\Xsz}^2 \,\dir 
   + \kappa_0 \| u(t) \|_{\Xsigz}^2 
   \le C,
   \quad \perogni t\ge 0,
\end{equation}
which implies $w \in L^2(0,\infty;\Xsz)$ (hence $u_t \in L^2(0,\infty;\Xsz')$ by \eqref{eq:w}) and $u \in L^\infty(0,\infty;\Xsigz)$. Furthermore, 
the right-continuity of $t \mapsto \J(u(t))$ and that of $t \mapsto u(t)$ (in the strong topology of $\Xsigz$) can be also proved as in~\cite{ASSe1}.

Now, it remains to derive \eqref{reg:beta} (which also implies \eqref{reg:gu}). We formally test \eqref{eq:u} by $\beta(u)$ and integrate it over the generic
interval $(t,t+T)$, $t\ge 0$, $T>0$. Owing to the monotonicity of 
$\beta$ (that is, $\langle\Asig u, \beta(u)\rangle_{\Xsigz} \geq 0$, formally), we obtain
 \begin{align}
  \lefteqn{
  \| \beta(u) \|_{L^2(t,t+T;L^2(\Omega))}^2
  }\nonumber\\
   & \le \int_t^{t+T} \big( w(r) + \lambda u(r) , \beta(u(r)) \big)\, \dir \nonumber \\
 \nonumber   
   & \le  \frac12 \| \beta(u) \|_{L^2(t,t+T;L^2(\Omega))}^2
   + \| w \|_{L^2(t,t+T;L^2(\Omega))}^2
   + \lambda^2 \| u \|_{L^2(t,t+T;L^2(\Omega))}^2\\
 \label{exi:12}
  & \le  \frac12 \| \beta(u) \|_{L^2(t,t+T;L^2(\Omega))}^2
   + C \| w \|_{L^2(0,\infty;\Xsz)}^2
   + C T \| u \|_{L^\infty(0,\infty;\Xsigz)}^2
\end{align}
(see also Appendix \S \ref{apdx:s:beta-est} for a rigorous derivation). Thus \eqref{reg:beta} follows, and it also provides in particular \eqref{reg:gu} and completes the proof. \qed

\subsection{Proof of Theorem~\ref{thm:reg}}%
Also in this case we just give formal estimates which can be made rigorous by approximation arguments (see Appendix \S \ref{apdx:thm3} for more details). 
In view of \eqref{eq:w} and \eqref{exi:11}, for any $t_0 > 0$ there exists
$t_1 \in (0,t_0)$ such that 
\begin{equation}\label{reg:11}
  \| u_t(t_1) \|_{\Xsz'}^2 = \| w(t_1) \|_{\Xsz}^2 \le C t_0^{-1}.
\end{equation}
Then, let us test \eqref{eq:w} by $w_t$. Let us also differentiate
\eqref{eq:u} in time and test the result by $u_t$. Summing the obtained
relations we then get
\begin{align}\nonumber
  & \frac12 \ddt \| w \|_{\Xsz}^2
   + \| u_t \|_{\Xsigz}^2
   + \io \beta'(u(x)) | u_t(x) |^2 \, \dix  
   = \lambda \| u_t \|_{L^2(\Omega)}^2\\
 \label{reg:12}
  & \mbox{}~~~~~~~~~~
  \le \frac12 \| u_t \|_{\Xsigz}^2
  + c \| u_t \|_{\Xsz'}^2
  \le \frac12 \| u_t \|_{\Xsigz}^2
  + c \| w \|_{\Xsz}^2,
\end{align}
thanks also to Ehrling's lemma and to the properties of $\As$. Then,
integrating over $(t_1,t)$, and using \eqref{eq:w}, \eqref{exi:11} and \eqref{reg:11}, we infer (by $t_0 > t_1$) that
\begin{equation}\label{reg:13}
  \| w(t) \|_{\Xsz}^2
   + \| u_t(t) \|_{\Xsz'}^2
   + \int_{t_0}^t \| u_t(r) \|_{\Xsigz}^2 \, \d r
  \le C ( 1 + t_0^{-1} )
\end{equation}
for all $t \geq t_0 > 0$. Here we note that $C$ above is independent of $t$ (and any final time $T$).
This implies \eqref{reg:u2}-\eqref{reg:w2}, as desired. \qed

\subsection{Proof of Theorem \ref{thm:ener}}

First, recall by \eqref{eq:ubeta} that 
\begin{equation}\label{der-H}
 \Asig u + \beta(u) = w + \lambda u.
\end{equation}
In case \eqref{reg:u3}, we refer the reader to~\cite[\S 4.6]{ASSe1}. In case \eqref{reg:u4} holds, one can apply a standard chain-rule for subdifferential operators in Hilbert spaces (see~\cite{HB1}) 
to (an $L^2$-extension of) the convex part of the energy functional defined on $H_0$,
$$
\phi(u) := \begin{cases}
	    \frac 1 2 \|u\|_{\Xsigz}^2 + \int_\Omega \betaciapo(u(x)) \, \d x &\mbox{ if } \ u \in \Xsigz \ \mbox{ and } \ \betaciapo(u(\cdot)) \in L^1(\Omega),\\
	    +\infty &\mbox{ otherwise}
	   \end{cases} 
$$
for $u \in H_0$. Here $\betaciapo$ is a primitive function of $\beta$, i.e., $\partial \betaciapo = \beta$, and it is lower semicontinuous and convex. 
Then by means of \eqref{reg:u4} and $\partial \phi(u) = \Asig u + \beta(u) \in L^2(0,T;H_0)$ by \eqref{der-H}, one deduces that $t \mapsto \phi(u(t))$ is absolutely continuous on $[0,T]$ and that
$$
\left( \Asig u(t) + \beta(u(\cdot,t)), u_t(t) \right)
= \left( \partial \phi(u(t)), u_t(t) \right) = \dfrac \d {\d t} \phi(u(t)) \ \mbox{ for a.e. } t \in (0,T),
$$
where we also used the fact that $\partial \phi (u)$ coincides with $\Asig u + \beta(u)$. Hence the assertion follows immediately. \qed

\subsection{\bf Proof of Proposition \ref{P:hoelder}}
We shall bootstrap regularity for $u$ by viewing equation~\eqref{eq:u} as a time-dependent family of elliptic problems, i.e.,
\begin{equation}\label{eq:u:ell}
  \Asig u + \beta(u) = f \ \mbox{ in } \Xsigz',
\end{equation}
where we have set $f:=\lambda u + w$.
Then, we shall determine which is the highest exponent
$p$ for which we can prove
\begin{equation}\label{boot:10}
  \| f(t) \|_{L^p(\Omega)}
    \le C,
\end{equation}
at least for large $t$. Correspondingly, from the fact that
\begin{equation}\label{b-leq-f}
 \|\beta(u(t))\|_{L^p(\Omega)} \leq   \| f(t) \|_{L^p(\Omega)}
\end{equation}
by the monotonicity of $\beta$ (see Appendix \S \ref{apdx:ss:bu}), one derives that
\begin{equation}\label{boot:1b}
  \| u(t) \|_{\Xp}
    \le C.
\end{equation}

We shall prove in fact that \eqref{boot:10} and hence \eqref{boot:1b} hold for $p = s^*$, where $s^*$ is given by
$$
s^*=\frac{2N}{N-2s}.
$$
Let us start with considering the case when $s < N/2$, which is the most difficult one (and, also, it always occurs when $N\ge 2$). Then, from \eqref{reg:w2} and Sobolev's embeddings we have
\begin{equation}\label{boot:11}
  \| w(t) \|_{L^{s^*}(\Omega)}
  \le C \quad \mbox{ for all } \ t \geq 1.
\end{equation}
We shall prove that also $\lambda u$ has the same summability.
Indeed, from 
\eqref{reg:u} we know that 
\begin{equation}\label{boot:12}
  \| u(t) \|_{L^{\sigma^*}(\Omega)}
  \le C,
  \quad \sigma^*:=\frac{2N}{N-2\sigma},
  \quad \mbox{ for all } \ t \geq 0.
\end{equation}
provided $\sigma < \frac{N}2$ (of course, for $\sigma \ge \frac{N}2$,
we have better). Now, if $\sigma^*\ge s^*$ (or, in other words, $\sigma\ge s$),
we reach the conclusion. 

So, let us assume $\sigma^*< s^*$ (or, equivalently, $\sigma< s$).
Then, we may apply:
\begin{lemma}\label{L:ellip}
 Let $p\in [2,\infty)$ and assume\/ \eqref{boot:10}. Then
 the solution $u$ to \eqref{eq:u:ell} satisfies
 \begin{equation}\label{reg:u:ell}
   u \in W^{\frac{2\sigma}p,p}(\RN),
    \quad \| u \|_{W^{\frac{2\sigma}p,p}(\RN)} \le c \| f \|_{L^p(\Omega)}
 \end{equation}
\end{lemma}
\noindent
(see Appendix \S \ref{apdx:ss:L:ellip} for a proof). Thanks to Sobolev's embeddings, \eqref{reg:u:ell} implies in 
particular
\begin{equation}\label{boot:13}
  u \in L^{\frac{Np}{N-2\sigma}}(\RN), 
   \quad \| u \|_{L^{\frac{Np}{N-2\sigma}}(\RN)} \le c \| f \|_{L^p(\Omega)}.
\end{equation}
Now, we may apply the above lemma starting, say, from $p=p_0=2$. Then, in accordance
with \eqref{boot:12}, we arrive at the first step to $p_1=\sigma^*:=\frac{2N}{N-2\sigma} = \frac{N}{N-2\sigma} p_0 > p_0$.
We may go on until, after a finite number $k$ of steps, $p_k \ge s^*$, as desired.
Notice that we cannot go on with iterations because
the regularity of $f$ has an upper threshold in view
of~\eqref{boot:11} (in other words, we cannot improve the summability 
of $w$). This completes the proof.

As a consequence, we have \eqref{boot:1b} for $p=s^*$ by (iii) of Proposition \ref{new:P0}. Then, we may also apply (ii) of Proposition \ref{new:P0} with that choice of $p$ provided that $p=s^*> N/(2\sigma)$, which corresponds exactly to \eqref{big:ssig}. The desired conclusion is proved. \qed
%


\section{Proof of Lemma \ref{L:omega}}
\label{sec:omega}

This section is devoted to proving Lemma \ref{L:omega}. Here an additional difficulty resides in the lack of regularity of weak solutions 
(particularly from the gap between $\Xsz$ and $\Xsigz$ by $s \neq \sigma$, see~\cite{ASSe1} for more details), compared to the classical
Cahn-Hilliard equation. Indeed, from the definition of weak solutions, one cannot directly deduce energy equalities (or inequalities) 
which could be exploited to prove the assertion. However, such a defect is compensated by the existence-uniqueness part (see Theorem \ref{thm:exi} 
and~\cite{ASSe1}), where several energy inequalities have already been established 
through a construction of weak solutions. Another 
difficulty lies on our rather general choice of $g$. In particular, we do not impose 
here any growth condition on $g$ (equivalently, on $\beta$), 
and hence, we need an extra argument to estimate the nonlinear term $\beta(u)$. To this end, we shall in fact employ \eqref{reg:beta}.

First, we recall \eqref{ei-i}, that is,
$$
  \int^t_0 \|w(r)\|_{\Xsz}^2 \,\d r
   + \J(u(t)) \leq \J(u_0) \quad \mbox{ for all } \ t \geq 0.
$$
Thanks to \eqref{g4-12}, we deduce that
\begin{equation}\label{wu-bdd}
 \int^\infty_0 \|w(r)\|_{\Xsz}^2 \,\d r
  + \sup_{t \geq 0} \| u(t) \|_{\Xsigz}^2
  \leq C,
\end{equation}
with a constant $C \geq 0$ independent of $t$ (but depending on $\Esi(u_0)$).
From equation~\eqref{eq:w}, using the relation (see Appendix \S \ref{apdx:ss:star}),
\begin{equation}\label{star}
  \|v\|_{\Xsz}^2 = \|\As v\|_{\Xsz'}^2 \quad \mbox{ for all } \ v \in \Xsz,
\end{equation}
we also have
$$
  \int^\infty_0 \|u_t(r)\|_{\Xsz'}^2 \,\d r \leq C.
$$
Now, let us fix an arbitrary sequence $t_n \to \infty$. Then
$$
a_n := \int^{t_n}_{t_n-1} \|u_t(r)\|_{\Xsz'}^2 \, \d r \to 0,
$$
which together with \eqref{reg:beta} implies
$$
a_n^{-1} \int^{t_n}_{t_n-1} \|u_t(r)\|_{\Xsz'}^2 \, \d r + \int^{t_n}_{t_n-1} \|\beta(u(r))\|_{L^2(\Omega)}^2 \, \d r \leq C.
$$
Then, there exists a sequence $\tau_n \in [t_n-1, t_n)$ such that
$$
a_n^{-1} \|u_t(\tau_n)\|_{\Xsz'}^2 + \|\beta(u(\tau_n))\|_{L^2(\Omega)}^2 \leq C.
$$
Thus we infer that, up to a non-relabeled subsequence of $n$,
\begin{align}\label{co:11}
  u_t(\tau_n) \to 0 & \quad \mbox{ strongly in } \Xsz',\\
 \label{co:12} 
  \beta(u(\tau_n)) \to \xi & \quad \mbox{ weakly in } L^2(\Omega)
\end{align}
for some function $\xi\in L^2(\Omega)$.
Then, using \eqref{star} with equation~\eqref{eq:w}, we also obtain
\begin{align}\label{co:13}
  \As w(\tau_n) \to 0 \quad &\mbox{ strongly in } \Xsz',\\
 \label{co:14} 
  w(\tau_n) \to 0 \quad &\mbox{ strongly in } \Xsz.
\end{align}
Moreover, since $\Xsigz$ is compactly embedded in $L^2(\Omega)$ for any $\sigma > 0$ (see Proposition \ref{P:embed}), up to a subsequence, one derives from \eqref{wu-bdd} that
\begin{align}\label{co:15}
  u(\tau_n) \to \phi \quad &\mbox{ weakly in } \Xsigz,\\
 \label{co:16}
   & \mbox{ strongly in } L^2(\Omega),\\
 \label{co:17}
   \Asig u(\tau_n) \to \Asig \phi \quad &\mbox{ weakly in }
       \Xsigz',
\end{align}
with some $\phi \in \Xsigz$. Therefore one obtains $\xi = \beta(\phi)$ by the demiclosedness of maximal monotone operators (see, e.g.,~\cite{HB1})
along with \eqref{co:12} and \eqref{co:16}, and moreover, we deduce that
\begin{equation}\label{stapro1}
  \lim_{n\nearrow \infty} \big( \beta(u(\tau_n)), u(\tau_n) \big)
    = \big( \beta(\phi), \phi\big).
\end{equation}
On the other hand, combining the fact that
$$
  \Asig u(\tau_n) + \beta(u(\tau_n)) = w(\tau_n) + \lambda u(\tau_n) \to \lambda \phi \quad \mbox{ strongly in } L^2(\Omega)
$$
and \eqref{co:12} (with $\xi = \beta(\phi)$) and \eqref{co:17}, $\Asig \phi + \beta(\phi) = \lambda \phi$ (in $L^2(\Omega)$). In particular, $\phi$ turns out to be a weak solution of the stationary problem, i.e., $\phi$ solves
\begin{equation}\label{stapro}
 \phi \in \Xsigz \quad \mbox{ and } \quad
  \Asig \phi + g(\phi) = 0 \ \mbox{ in } \Xsigz'.
\end{equation}
Moreover, we observe by \eqref{stapro1} that
\begin{align}
 \lim_{n \to \infty} \|u(\tau_n)\|_{\Xsigz}^2
 &= \lim_{n \to \infty} \left( w(\tau_n) + \lambda u(\tau_n) - \beta(u(\tau_n)), u(\tau_n) \right)\nonumber\\
 &= (-g(\phi),\phi)
 = \|\phi\|_{\Xsigz}^2.
 \label{stapro2}
\end{align}
Relation \eqref{stapro2}, together with \eqref{co:15} and the uniform convexity of $\Xsigz$, implies
$$
  u(\tau_n) \to \phi \quad \mbox{ strongly in } \Xsigz.
$$
By definition of subdifferential and \eqref{co:16}, we also find that
\begin{align*}
\limsup_{n \to \infty} \int_\Omega \betaciapo(u(\tau_n)) \, \d x
&\leq \int_\Omega \betaciapo(\phi) \, \d x
+ \lim_{n \to \infty} \int_\Omega \beta(u(\tau_n)) \left(u(\tau_n)-\phi\right)\, \d x\\
&= \int_\Omega \betaciapo(\phi) \, \d x,
\end{align*}
which together with the lower semicontinuity of $\betaciapo$ entails
$$
\lim_{n \to \infty} \int_\Omega \betaciapo(u(\tau_n)) \, \d x =
\int_\Omega \betaciapo(\phi) \, \d x.
$$
 Combining all these facts, we deduce (by $\gciapo(s) = \betaciapo(s) - (\lambda/2)s^2$ from \eqref{hp:g3}) that
 $$
 \J(u(\tau_n)) \to \J(\phi).
 $$
 Now, let us notice that $\J(u(\cdot))$ is nonincreasing. Hence for general $t_n \to \infty$, one also obtains
$$
\lim_{t_n \to \infty} \J(u(t_n)) = \J(\phi).
$$ 

We further observe that
\begin{align*}
  \|u(t_n) - \phi\|_{\Xsz'}
  & \leq \int^{t_n}_{\tau_n} \|\partial_\tau u(\tau)\|_{\Xsz'} \, \d \tau
  + \|u(\tau_n) - \phi\|_{\Xsz'}
 \\
  & \leq \left(\int^\infty_{\tau_n} \|\partial_\tau u(\tau)\|_{\Xsz'}^2 \, \d \tau\right)^{1/2} \sqrt{t_n-\tau_n}
  + \|u(\tau_n) - \phi\|_{\Xsz'}
 \\
  & \to 0.
\end{align*}
Thus $u(t_n) \to \phi$ strongly in $\Xsz'$.
Furthermore, since $u(t_n)$ is bounded in $\Xsigz$,
we also find that, along a (not relabeled) subsequence, $u(t_n) \to \phi$ strongly in $L^2(\Omega)$.
Noting that
\begin{align*}
\dfrac 1 2 \|u(t_n)\|_{\Xsigz}^2
&= \J(u(t_n)) - \int_\Omega \betaciapo(u(t_n)) \, \dx
+ \dfrac \lambda 2 \|u(t_n)\|_{L^2(\Omega)}^2
\end{align*}
and recalling that $u(t_n) \to \phi$ strongly in $L^2(\Omega)$ and weakly in $\Xsigz$, we see that
\begin{align*}
\dfrac 1 2 \limsup_{n \to \infty} \|u(t_n)\|_{\Xsigz}^2
&\leq \lim_{n \to \infty} \J(u(t_n)) - \liminf_{n \to \infty}\int_\Omega \betaciapo(u(t_n)) \, \dx
+ \dfrac \lambda 2 \lim_{n \to \infty}\|u(t_n)\|_{L^2(\Omega)}^2\\
&\leq \J(\phi) - \int_\Omega \betaciapo(\phi) \, \dx 
+ \dfrac \lambda 2 \|\phi\|_{L^2(\Omega)}^2
= \dfrac 1 2 \|\phi\|_{\Xsigz}^2,
\end{align*}
which along with the uniform convexity of $\Xsigz$ yields
$$
u(t_n) \to \phi \quad \mbox{ strongly in } \Xsigz.
$$
%
%
%
%
 %
 This completes the proof. \qed


\section{Proof of Theorem \ref{T:LSI-bdd}}
\label{sec:Loj-2}

In this section, we shall give a proof of Theorem \ref{T:LSI-bdd}, which provides a \L ojasiewicz-Simon inequality for fractional Laplacian. Due to a defect of regularity 
property for the fractional Dirichlet Laplacian, one needs to modify the standard
arguments of proofs for \L S inequalities (see Introduction). For instance, the (classical) Laplace 
operator defined over $L^r(\Omega)$ (for $r \in (1,\infty)$) with the homogeneous Dirichlet boundary condition has a regular domain, namely, 
$D(-\Delta) = W^{2,r}(\Omega) \cap H^1_0(\Omega)$, and moreover, this property (particularly for $r > 0$ large enough) plays a crucial role in 
the proof in~\cite{FeiSim00} (cf.~Schauder theory plays a similar role in~\cite{RyHo}). However, the fractional Laplace operator $\Ds$ defined 
on $L^r(\Omega) \simeq L^r_0(\R^N)$ may fail to fulfill corresponding properties, e.g., $D(\Ds) = W^{2s,r}(\Omega) \cap \Xsz$ (see~\cite{RoSe12,RoSe13} for some counterexamples).

Concerning the cases (\nii) and (\shi) of Theorem \ref{T:LSI-bdd}, we replace $g(\cdot)$ with a function $\tilde g(\cdot) \in C^1(\R)$ satisfying
\begin{equation}\label{tg-ii}
\tilde{g}(s) = g(s) \quad \mbox{ if } \ |s| < \gamma \vee \eta \quad \mbox{ and } \quad |\tilde{g}(s)| \leq M \quad \mbox{ if } \ |s| > (\gamma \vee \eta) + 1 
\end{equation}
for some constant $M$ large enough. Then we denote by $\tJ$ the energy functional $\J$ whose potential part $g$ is replaced by 
the modified one $\tilde g$. Here and henceforth, we simply write $g$ and $\J$ instead of $\tilde g$ and $\tJ$, respectively, if no confusion may arise. Let us start with the following:
\begin{lemma}
 In any of the cases {\rm (i)--(iv)} of Theorem \ref{T:LSI-bdd}, $\J$ is of class $C^2$ in $\Xsigz$.
 \end{lemma}

 \begin{proof}
  In the case of (\ichi) and (\san), due to (H3), the functional
  $$
  G(u) = \int_\Omega \widehat g(u(x)) \, \d x
  $$
  is of class $C^2$ in $\Xsigz$ (see (iii) of Remark \ref{R:analytic}). In the other cases, i.e., (\nii) and (\shi), the modified function $\tilde g$ satisfies \eqref{hypo-sig1} (then \eqref{g-ciappo} and \eqref{g} as well), and hence, $\J$ with $g$ replaced by $\tilde g$ also has $C^2$ regularity.
 \end{proof}

\begin{remark}\label{R:g-modif}
 {\rm
We shall derive a \L ojasiewicz-Simon inequality for $\J$ with the modified function $\tilde g(\cdot)$; then the modification of $g(\cdot)$ defined above will be needed to guarantee the $C^2$ regularity of the energy functional $\J$ in $\Xsigz$. On the other hand, the difference between $g(\cdot)$ and $\tilde g(\cdot)$ cannot be neglected; indeed, we shall apply the classical \L ojasiewicz inequality (see Proposition \ref{P:L}) to a function derived from $\J$ with $\tilde g$ (see \eqref{H} below) defined on a finite dimensional space, and then, all constants appeared in the \L ojasiewicz inequality may depend on the modified function $\tilde g(\cdot)$ itself in an indefinite way.
 }
\end{remark}

We are ready to give a proof of Theorem \ref{T:LSI-bdd}. This proof is divided into several steps. Define the linearized operator $\mathscr L(\phi) : \Xsigz \to \Xsigz'$ of $\J'$ at an equilibrium $\phi \in \Xsigz \cap L^\infty(\Omega)$ by
$$
\mathscr L(\phi) u := \J''(\phi) u = \Asig u + g'(\phi) u
\quad \mbox{ for } \ u \in \Xsigz.
$$
Then since $g'(\phi) \in L^\infty(\Omega)$, by Fredholm alternative, one finds that the null set
$$
\mathcal N := \mathrm{Ker} (\mathscr L(\phi))
= \left\{ v \in \Xsigz \colon \mathscr L(\phi)v = 0 \right\}
$$
is finite dimensional (see, e.g.,~\cite[Theorem IX.23]{BrFA}). For latter use, let us consider the linearized problem,
\begin{equation}\label{L-eqn}
\mathscr L(\phi) u = h
\end{equation}
for some $h \in L^p(\Omega)$ and $p \geq 2$ (set $h = 0$ for $u \in \mathcal
 N$). 

\begin{propo}\label{new:P1}
Let $p \geq 2$ and let $u \in \Xsigz$ be a solution of \eqref{L-eqn} with $h \in L^p(\Omega)$. Then $u$ belongs to $\Xp$. In particular, it follows that $\mathcal N \subset \Xp$ for any $p \in [2,\infty)$.
\end{propo}

 \begin{proof}
As in Lemma \ref{L:ellip}, (formally) test \eqref{L-eqn} by $|u|^{p-2}u$ to see that
\begin{align*}
\omega_0 [u]^p_{W^{\frac{2\sigma}p, p}(\Omega)} 
&\leq \int_\Omega h |u|^{p-2}u ~\dx - \int_\Omega g'(\phi) |u|^p ~
 \dx\\
&\leq \|h\|_{L^p(\Omega)} \|u\|_{L^p(\Omega)}^{p-1} + C \|u\|_{L^p(\Omega)}^p\\
&\leq C \|h\|_{L^p(\Omega)}^p + C \|u\|_{L^p(\Omega)}^p
\end{align*}
for some constant $\omega_0 > 0$. 
By using the compact and continuous embedding $W^{\frac{2\sigma}p, p}(\Omega) \hookrightarrow L^p(\Omega)$
and the continuous embedding $L^p(\Omega) \hookrightarrow L^2(\Omega)$ (recall that $p\ge 2$) 
along with Ehrling's lemma, for arbitrarily small $\vep > 0$ one can take $C_\vep > 0$ such that
$$
\dfrac{\omega_0}2 [u]^p_{W^{\frac{2\sigma}p, p}(\Omega)} 
\leq C \|h\|_{L^p(\Omega)}^p + \vep \|u\|_{W^{\frac{2\sigma}p,p}(\Omega)}^p + C_\vep \|u\|_{L^2(\Omega)}^p,
$$
which together with Poincar\'e's inequality (see Proposition \ref{apdx:P:poincare}) and~\cite[\S 6]{Dine_Pala_Vald} implies
$$
u \in L^{p^*(\sigma)}(\Omega) \quad \mbox{ with } \ 
p^*(\sigma) := \dfrac{Np}{(N - 2\sigma)_+} > p.
$$
Recalling \eqref{L-eqn} along with \eqref{hp:g1} and the fact $g'(\phi) \in L^\infty(\Omega)$ by assumption, we observe that
$$
\Asig u = h - g'(\phi) u \in L^p(\Omega),
$$
which entails $u \in \Xp$ by Proposition \ref{new:P0}.
In particular, if $h = 0$, then one can carry out the argument above for any $p \in [2,\infty)$. Thus we deduce that $u \in \Xp(\Omega)$ for any $p \in [2,\infty)$.
 \end{proof}

Let $P : L^2(\Omega) \to \mathcal N$ be the projection in $L^2(\Omega)$ onto $\mathcal N$. Then we claim that
\begin{claim}\label{new:claim1}
 $\mathscr L(\phi) + P$ is a linear isomorphism {\rm (}=bijective bicontinuous mapping{\rm )} from $\Xsigz$
 to $\Xsigz'$.
\end{claim}

\begin{proof}
We note that
\begin{align*}
 \Asig^{-1} \left( \mathscr L(\phi) + P \right)
&= \Asig^{-1} \left( \Asig + g'(\phi) + P \right)\\
&= \Id + \Asig^{-1}\left(g'(\phi) + P\right) 
: \Xsigz \to \Xsigz,
\end{align*}
where $\Id$ denotes the identity mapping in $\Xsigz$ 
and $\Asig^{-1} : \Xsigz' \to \Xsigz$ stands
for the inverse mapping of $\Asig$ (it is well defined by Proposition \ref{P:iso}. See~\cite{ASSe1}). Set
$$
T := -\Asig^{-1}\left(g'(\phi) + P\right) 
: \Xsigz \to \Xsigz.
$$
Then $T$ is bounded. We shall show that $T$ is compact in $\Xsigz$. Indeed,
let $(f_n)$ be a bounded sequence in $\Xsigz$. Then by $g'(\phi) \in L^\infty(\Omega)$,
$$
\left\| (g'(\phi) + P) f_n \right\|_{L^2(\Omega)}
\lesssim \|f_n\|_{L^2(\Omega)} \lesssim \|f_n\|_{\Xsigz} \leq C,
$$
which implies that $(g'(\phi) + P) f_n$ is precompact in $\Xsigz'$ (see Proposition \ref{P:embed}). Hence since $\Asig^{-1}$ is an isomorphism from $\Xsigz'$ to $\Xsigz$ (see Proposition \ref{P:iso}), we find that
$u_n := Tf_n = - \Asig^{-1} ((g'(\phi) + P) f_n)$ is precompact in
$\Xsigz$. Thus $T$ is compact in $\Xsigz$.

By the Fredholm alternative, we also observe that
\begin{equation}\label{Fredholm}
\mathrm{Ker}(\Id-T) = \{0\} \quad \Leftrightarrow \quad
\mathrm{Rg}(\Id-T) = \Xsigz.
\end{equation}
So we shall prove that
$\mathrm{Ker}(\Id-T) = \{0\}$. Let $u \in D(T) = \Xsigz$ satisfy
\begin{equation}\label{resol}
(\Id - T) u = 0, \quad \mbox{ i.e., } \ \mathscr L(\phi)u+Pu=0. 
\end{equation}
 Decompose the above $u \in \Xsigz \subset L^2(\Omega)$ as $u = u^0 + u^\bot$ for $u^0 \in
\mathcal N$ and $u^\bot \in \mathcal N^\bot$. Then it follows from \eqref{resol} that
\begin{equation}\label{Lu0}
\mathscr L(\phi) u^\bot + u^0 = 0.
\end{equation}
Test it by $u^0$ to get
$$
\left( \mathscr L(\phi) u^\bot, u^0 \right)_{L^2(\Omega)} + \|u^0\|_{L^2(\Omega)}^2 = 0,
$$
which together with the symmetry of $\mathscr L(\phi)$ gives
$\|u^0\|_{L^2(\Omega)} = 0$. Hence by \eqref{Lu0} $u^\bot$ belongs to $\mathcal N$. Due to
the fact that $u^\bot \in \mathcal N^\bot$, we deduce that $u^\bot = 0$.
Thus $u = 0$, and therefore, $\mathrm{Ker}(\Id-T) = \{0\}$. Furthermore, \eqref{Fredholm} implies the surjectivity of $\Id - T$. The continuity of $(\Id - T)^{-1}$ follows 
from the continuity (boundedness) of $\Id-T$ as well as Open Mapping Theorem.

Since $\Asig : \Xsigz \to \Xsigz'$ is an
isomorphism, we conclude that $\mathscr L(\phi) + P = \Asig (\Id-T)$
is also an isomorphism.  Thus the claim has been proved.
\end{proof}

We next claim that

\begin{claim}\label{new:claim2}
For $p \geq 2$,
 $\mathscr L(\phi) + P$ is a linear isomorphism from $\Xp$ to $L^p(\Omega)$.
\end{claim}

\begin{proof}
 It is sufficient to prove the surjectivity; indeed, the injectivity follows from Claim \ref{new:claim1}. For $h \in L^p(\Omega) \subset L^2(\Omega)$, one can decompose $h$ as $h = h_1 + h_2$ for some $h_1 \in \mathcal N^\bot$ and $h_2 \in \mathcal N$. Then since $h_1 \in \mathcal N^{\bot}\cap L^p(\Omega)$ (note that $h_2 \in \mathcal N \subset \Xp \subset L^p(\Omega)$ by Claim \ref{new:P1}), one can take $u_1 \in \mathcal N^\bot \cap \Xp$ such
 that
 \begin{equation}\label{Lu1=h1}
  \mathscr L(\phi) u_1 = h_1
 \end{equation}
 (see Appendix \S \ref{apdx:ss:Lu1}). Hence
\begin{align*}
 h = h_1 + h_2 
 &= \mathscr L(\phi) u_1 + Ph_2\\
 &= \mathscr L(\phi) (u_1+h_2) + P(u_1 + h_2)
 = (\mathscr L(\phi) + P)(u)
\end{align*}
for $u := u_1 + h_2 \in \Xp$. Thus $\mathscr L(\phi) + P$ is surjective from $\Xp$ to $L^p(\Omega)$.

For any $u \in \Xp$, it holds that
$$
\|\mathscr L(\phi) u + Pu\|_{L^p(\Omega)}
 \leq \|\mathscr L(\phi) u\|_{L^p(\Omega)} + \|Pu\|_{L^p(\Omega)}.
$$
Since $\dim \mathcal N$ is finite, we have
$$
\|Pu\|_{L^p(\Omega)} \lesssim \|Pu\|_{L^2(\Omega)}
\leq \|u\|_{L^2(\Omega)} \lesssim \|u\|_{L^p(\Omega)}
 \quad \mbox{ for all } \ u \in \Xp.
 $$
Here we used equivalence of (arbitrary) norms in finite dimensional spaces, boundedness of $P$ and H\"older's inequality. Moreover, it follows that
$$
 \|\mathscr L(\phi) u\|_{L^p(\Omega)}
 \leq \|\Asig u\|_{L^p(\Omega)} + \|g'(\phi)u\|_{L^p(\Omega)}\\
 \leq \|u\|_{\Xp} + C\|u\|_{L^p(\Omega)}
 $$
 for all $u \in \Xp$. Thus $\mathscr L(\phi) + P$ is bounded linear from $\Xp$ to $L^p(\Omega)$. 
 By the Open Mapping Theorem, $(\mathscr L(\phi)+P)^{-1} : L^p(\Omega) \to \Xp$ is also bounded.
\end{proof}

\begin{propo}\label{new:Panal}
 For $p > \frac{N}{\sigma}$, the operator $\J' + P : \Xp \to L^p$ is analytic in a neighborhood of $\phi$ in $\Xp$.
\end{propo}

\begin{proof}
Let us start with calculating the derivative of the map $\Asig : \Xsigz \to \Xsigz'$,
$$
\left\langle \Asig'(v) e, w \right\rangle_{\Xsigz}
= \dfrac{C_\sigma} 2 \iint_{\mathbb R^{2N}}
\dfrac{(e(x)-e(y))(w(x)-w(y))}{|x-y|^{N+2\sigma}} ~\dx \, \dy
= \left\langle \Asig e, w \right\rangle_{\Xsigz}
$$
for any $e,v,w \in \Xsigz$. Therefore $\Asig^{(n)} \equiv 0$ for $n \geq 2$, and particularly, $\Asig$ is analytic in $\Xsigz$. Indeed, one observes that
$$
 \left\langle \Asig(u+e),w \right\rangle_{\Xsigz}
 = \left\langle \Asig u, w \right\rangle_{\Xsigz}
 + \left\langle \Asig e, w \right\rangle_{\Xsigz}
 = \left\langle \Asig u, w \right\rangle _{\Xsigz}
 +  \left\langle \Asig'(u)e, w \right\rangle_{\Xsigz}
$$
 for any $u,e,w \in \Xsigz$. Hence $\Asig(u+e) = \Asig u + \Asig'(u)e$ in $\Xsigz'$ for $u,e \in \Xsigz$. In a similar way, one can also prove that $P : L^2(\Omega) \to \mathcal N$ is analytic in $L^2(\Omega)$, and moreover, $P(u+e) = Pu + P'(u)e = Pu + Pe$ and $P^{(n)} \equiv 0$ for $n \geq 2$. By virtue of the embeddings $\Xp \subset \Xsigz \subset L^2(\Omega)$, one can check the analyticity (in $\Xp$) of the restrictions of $\Asig$ and $P$ onto $\Xp$. Indeed, we have, for $v, h \in \Xp$,
 $$
 \|P'(v)h\|_{L^p(\Omega)} = \|Ph\|_{L^p(\Omega)} \lesssim \|Ph\|_{L^2(\Omega)}
 \leq \|h\|_{L^2(\Omega)} \lesssim \|h\|_{\Xp}
 $$
 and
 $$
 \|\Asig'(v)h\|_{L^p(\Omega)}
 = \|\Asig h\|_{L^p(\Omega)} \leq \|h\|_{\Xp},
 $$
 whence we deduce that $P'(v), \Asig'(v) \in \mathcal L(\Xp, L^p(\Omega))$. Moreover, we recall that $P^{(n)} \equiv 0$ and $\Asig^{(n)} \equiv 0$ for $n \geq 2$. Thus we infer that the mapping $u \mapsto \Asig u + Pu$ from $\Xp$ to $L^p(\Omega)$ is analytic (in $\Xp$) (see \S \ref{Ss:analytic}). So it remains to prove the analyticity of the map $u \mapsto g(u)$ from $\Xp$ to $L^p(\Omega)$.

In the case of (\ichi), let $v \in \Xp$ be fixed and let $h \in \Xp$ be such that $\|h\|_{\Xp} \leq r$ for $r > 0$. Then recalling the embedding $\Xp \hookrightarrow L^\infty(\RN)$ by $p > N/(2\sigma)$ (see Proposition \ref{new:P0}), we note that
\begin{equation}\label{h-leq-r}
 \|h\|_{L^\infty(\RN)} \leq C_{p,\sigma} \|h\|_{\Xp} \leq C_{p,\sigma} r.
\end{equation}
 We choose $r > 0$ such that $M C_{p,\sigma} r < 1$,
 where $M$ is the constant appearing in (H1). Hence by Remark \ref{R:analytic} and $C_{p,\sigma} r < M^{-1}$, we deduce that, for every $x \in \Omega$,
 $$
 g(v(x)+h(x)) = g(v(x)) + \sum_{n=1}^\infty \dfrac{g^{(n)}(v(x))}{n!}h(x)^n,
 $$
 where the series of the right-hand side is convergent uniformly in $\Omega$. Let $T : v \mapsto g(v(\cdot))$ be a mapping from $\Xp$ to $L^p(\Omega)$ and set
 \begin{equation}\label{def-Tn}
 T_n(v) [h_1,\ldots,h_n] := \dfrac{g^{(n)}(v(x))}{n!}h_1(x) \cdots h_n(x)
\end{equation}
 for $v,h_1,\ldots,h_n \in \Xp$. Then by (H1) with $a = b = \infty$, we derive that
 \begin{align*}
  \|T_n(v)\|_{\mathcal L^n(\Xp,L^p(\Omega))}
  &= \sup_{\|h_j\|_{\Xp} = 1} \left\| \dfrac{g^{(n)}(v(x))}{n!}h_1(x) \cdots h_n(x) \right\|_{L^p(\Omega)}\\
  &\leq C M^n  \sup_{\|h_j\|_{\Xp} = 1} \left\| h_1(x) \cdots h_n(x) \right\|_{L^p(\Omega)}\\
  &\leq C M^n |\Omega|^{1/p}  \sup_{\|h_j\|_{\Xp} = 1} \| h_1(x) \|_{L^\infty(\Omega)} \cdots \|h_n(x)\|_{L^\infty(\Omega)}\\
  &\leq C M^n |\Omega|^{1/p} C_{p,\sigma}^n.
 \end{align*}
Thus we have
 $$
 \sup_{n \in \N} \|T_n(v)\|_{\mathcal L^n(\Xp,L^p(\Omega))} r^n \leq C |\Omega|^{1/p} \sup_{n \in \N} (M C_{p,\sigma} r)^n < \infty,
 $$
 due to $0 < M C_{p,\sigma} r < 1$. Thus $T : \Xp \to L^p(\Omega)$ turns out to be analytic in $\Xp$, and therefore, so is $\J' + P$.

 In the case of (\san), we let $p > N/\sigma$ and take $v$ from an $\vep$-neighbourhood of $\phi$ in $\Xp$ (i.e., $\|\phi-u\|_{\Xp} < \vep$). 
 Moreover, let $h \in \Xp$ be such that $\|h\|_{\Xp} < r$. Then, the positive bounded equilibrium $\phi(x)$ satisfies
 \begin{equation}\label{bdry-Hoelder}
  \phi(x) \geq C_0 \mathrm{dist}(x,\partial \Omega)^\sigma \ \mbox{ for all  } \ x \in \overline\Omega
 \end{equation}
 for some $C_0 > 0$. Indeed, by~\cite[Theorem 1.2]{RoSe12} along with the fact that $g(\phi) \in L^\infty(\Omega)$ (see~\eqref{hp:g1}), we assure that
 $\phi(x)/\mathrm{dist}(x,\partial \Omega)^\sigma$ is continuously extended onto $\overline\Omega$ (and it is of class $C^\beta$ over $\overline\Omega$
 for some $0 < \beta < \min\{\sigma,1-\sigma\}$). On the other hand, since $\phi$ is positive in $\Omega$, it follows that
 $\phi(x)/\mathrm{dist}(x,\partial \Omega)^\sigma > 0$ for $x \in \Omega$. Moreover, we claim that $\phi(x)/\mathrm{dist}(x,\partial \Omega)^\sigma$ 
 is also positive for all $x \in \partial \Omega$. Indeed, we can rewrite \eqref{elliptic-f} as
 $$
 \Dsig \phi = c \phi
 $$
 with $c := -g(\phi)/\phi$. Then $c$ belongs to $L^\infty(\Omega)$, since $s \mapsto g(s)/s$ is continuous in $(0,\infty)$ and has a finite limit as $s \to 0_+$ by \eqref{hp:g1}. 
 Thus we can apply the fractional Hopf lemma (see~\cite[Lemma 1.2]{GrSe} and also Proposition \ref{P:Hopf} in Appendix \S \ref{S:hopf}) and verify 
 the positivity of $\phi(x)/\mathrm{dist}(x,\partial \Omega)^\sigma$ over $\partial \Omega$. Combining all these facts, we obtain \eqref{bdry-Hoelder}.
Hence by (ii) of Proposition \ref{new:P0} together with $p > N/\sigma$,
 \begin{align}
  v(x) + h(x) &\geq C_0 \mathrm{dist}(x, \partial \Omega)^\sigma - \|\phi - v - h\|_{C^\sigma (\overline{\Omega})} \mathrm{dist}(x, \partial \Omega)^\sigma\nonumber\\
  &\geq [C_0 - C ( \|\phi - v\|_{\Xp} + \|h\|_{\Xp})] \mathrm{dist}(x,\partial \Omega)^\sigma\nonumber\\
  &\geq [C_0 - C (\vep + r)] \mathrm{dist}(x,\partial \Omega)^\sigma \nonumber \\
  &=: \eta \,\mathrm{dist}(x,\partial \Omega)^\sigma > 0
  \quad \mbox{ for all } \ x \in \Omega,\label{v-Hopf}
 \end{align}
 provided that $\vep + r$ is small enough so that $\eta := C_0 - C(\vep + r) > 0$.

 Due to (H2) (with $b = \infty$) and \eqref{v-Hopf} (with $h \equiv 0$), we observe that
\begin{align*}
\left|
 \dfrac{g^{(n)}(v(x))}{n!} h_1(x) \cdots h_n(x)
\right|
&\leq 
 C M^n \left|\dfrac{h_1(x) \cdots h_n(x)}{v(x)^n} \right|\\
 &\leq 
 C M^n \widetilde C_{\sigma,p}^n \left|\dfrac{\|h_1\|_{\Xp} \cdots \|h_n\|_{\Xp}}{\eta^n} \right| \ \mbox{ for a.e. } \ x \in \Omega
\end{align*}
 for any $n \in \mathbb N$ and $h_j \in \Xp$ ($j = 1,2,\ldots,n$). Here we used the fact that
 $$
 |h_j(x)| \leq \|h_j\|_{C^\sigma(\overline\Omega)} \mathrm{dist}(x,\partial \Omega)^\sigma
 \leq \widetilde C_{\sigma,p} \|h_j\|_{\Xp} \mathrm{dist}(x,\partial \Omega)^\sigma
 $$
 for some constant $\widetilde C_{\sigma,p} > 0$ (see Proposition \ref{new:P0}). This implies that
 $$
   \|T_n(v)\|_{\mathcal L^n(\Xp,L^p(\Omega))} \leq C \dfrac{M^n}{\eta^n}  {\widetilde C_{\sigma,p}}^n,
 $$
 whence follows
 $$
 \sup_{n \in \N} \|T_n(v)\|_{\mathcal L^n(\Xp,L^p(\Omega))}  r^n < \infty,
 $$
 if $M r \widetilde C_{\sigma,p}/\eta < 1$. Moreover,
 $$
 g(v(x)+ h(x)) = g(v(x)) + \sum_{n = 1}^\infty \dfrac{g^{(n)}(v(x))}{n!} h(x)^n \quad \mbox{ for a.e. } x \in \Omega
 $$
 is uniformly convergent over $\Omega$, provided that $Mr\widetilde C_{\sigma,p}/\eta < 1$.
 Therefore $T : \Xp \to L^p(\Omega)$ is analytic at $v$, and hence, so is $T$ in the $\vep$-neighbourhood of $\phi$ in $\Xp$.
 
 In the case of (\nii), let $p > N/(2\sigma)$ and take $v$ from an $\vep$-neighbourhood of $\phi$ in $\Xp$ (i.e., $\|\phi - v\|_{\Xp} < \vep$) for $\vep > 0$ small enough. 
 Exploiting the embedding $\Xp \hookrightarrow L^\infty(\R^N)$ (by $p > N/(2\sigma)$) and choosing $\vep > 0$ small enough, by $\|\phi\|_{L^\infty(\Omega)} < \gamma < a \wedge b$, one observes that 
 $$
 \|v\|_{L^\infty(\Omega)} \leq \|\phi\|_{L^\infty(\Omega)} + C_{p,\sigma}\vep
 < a \wedge b.
 $$
We next let $h \in \Xp$ be such that $\|h\|_{\Xp} < r$ and take $r > 0$ small enough so that
$$
\|v + h\|_{L^\infty(\Omega)} < a \wedge b, \quad \|h\|_{L^\infty(\Omega)} < M^{-1} \quad \mbox{ and } \quad \|h\|_{\Xp} < M^{-1}.
 $$
Then (H1) implies
 $$
g(v(x)+ h(x)) = g(v(x)) + \sum_{n = 1}^\infty
 \dfrac{g^{(n)}(v(x))}{n!} h(x)^n
\quad \mbox{ for a.e. } x \in \Omega
 $$
 uniformly over $\Omega$ (see (i) of Remark \ref{R:analytic}). Repeating the same argument as in (\ichi), we conclude that $T$ is analytic at $v$; hence $T$ is analytic in the $\vep$-neighbourhood of $\phi$ in $\Xp$. So is $\J + P$.

In the case of (\shi), we take $v$ and $h$ and choose $\vep$ and $r$ small enough as in (\san). Then, one can also check that
 $$
 v(x) + h(x) < b \quad \mbox{ for all } \ x \in \Omega
 $$
 by taking $\vep > 0$ small enough. Repeating a similar argument to those of (\san) and (\nii), one can verify that $T : \Xp \to L^p(\Omega)$ is analytic at $v$, and hence, so is $T$ in the $\vep$-neighbourhood of $\phi$ in $\Xp$.
 \end{proof}

 The rest of proof runs as in~\cite{FeiSim00} (see also~\cite{Simon83}). However, for the convenience of the reader, we give a complete proof. Since $\J' + P : \Xsigz \to \Xsigz'$ is of class $C^1$, by Claim \ref{new:claim1}, one can apply a $C^1$ inverse function theorem to $\J' + P$ and ensure the existence of an inverse mapping,
 $$
 B = (\J'+P)^{-1} : U^* \to V^*
 $$
 of class $C^1$ from a neighborhood $U^*$ of $(\J'+P)(\phi) = P\phi$ in $\Xsigz'$ to a neighborhood $V^*$ of $\phi$ in $\Xsigz$. Furthermore, by the analytic inverse function theorem, since the map $\J' + P : \Xp \to L^p(\Omega)$ is analytic (at least in a small neighbourhood of $\phi$) for $p$ large enough, one can take a neighborhood $U_p \subset U^*$ of $P\phi$ in $L^p(\Omega)$ and a neighborhood $V_p \subset V^*$ of $\phi$ in $\Xp$ such that 
$$
 B = (\J'+P)^{-1} : U_p \to V_p \ \mbox{ is analytic in } U_p.
$$
 
Define a function $H: \mathcal N \cap U_p \to \mathbb R$ by
\begin{equation}\label{H}
H(u):= \J \circ B|_{\mathcal N}(u) \quad \mbox{ for } \ u \in \mathcal N
\cap U_p.
\end{equation}
Here we used a proper identification $\mathcal N \subset L^2(\Omega) \subset \Xsigz'$. From the analyticity of $\J$ on $\Xp$ (see Appendix \S \ref{apdx:s:analytic}), we deduce that $H$ is also analytic on $\mathcal N \cap U_p$. Let us observe that for $u \in \mathcal N \cap U_p$ and $v \in \mathcal
N ~ (\subset \Xsigz')$,
\begin{equation}\label{dG}
\langle H'(u), v \rangle_{\mathcal N} = \left\langle \J'(Bu), B'(u) v
\right\rangle_{\Xsigz}, 
\end{equation}
where $H'$ and $B'$ denote the Fr\'echet derivatives (i.e., gradients) of $H$ and $B$,
respectively (note that $B'$ maps from $U^*$ to $\mathcal
L (\Xsigz', \Xsigz)$; hence $B'(u) \in \mathcal
L (\Xsigz', \Xsigz)$ for $u \in U^p \subset U^*$). In
particular, substitute $u = P \phi \in \mathcal N \cap U_p$. Then since $\J'(\phi) = 0$, it follows that
\begin{align*}
\langle H'(P\phi), v \rangle_{\mathcal N} 
&= \left\langle \J'(B \circ P \phi), B'(P \phi) v
\right\rangle_{\Xsigz}\\
&= \left\langle \J'(B \circ (\J' + P) \phi), B'(P \phi) v
\right\rangle_{\Xsigz}\\
&= \left\langle \J'(\phi), B'(P \phi) v
\right\rangle_{\Xsigz} = 0
\quad \mbox{ for all } \ v \in \mathcal N,
\end{align*}
whence follows $H'(P\phi) = 0$ in $\mathcal N'$. Since $\mathcal N$ is
finite dimensional, one can apply the classical \L ojasiewicz inequality (see Proposition \ref{P:L}) to $H$ and obtain the following: there exist constants $\delta_0, C > 0$, $\theta \in (0,1/2]$ such that for all $n \in \mathcal N$,
\begin{equation}\label{L}
|H(n) - H(P\phi)|^{1-\theta} \leq C\|H'(n)\|_{\mathcal N'}
\end{equation}
whenever $\|n-P\phi\|_{\mathcal N} < \delta_0$ (it also implies $n \in
U_p$ by taking $\delta_0>0$ small enough by $\dim \mathcal N < \infty$). Here we also note that
\begin{equation}\label{HPphi}
H(P\phi) = \J(B \circ P\phi) = \J(\phi).
\end{equation}

Now, let $u \in \Xsigz$ satisfy
\begin{equation}\label{u-phi}
\|u - \phi\|_{\Xsigz} < \delta,
\end{equation}
for $\delta > 0$. Then, it holds that
\begin{equation}\label{PuPphi}
\|Pu - P\phi\|_{\mathcal N} \lesssim \|Pu-P\phi\|_{L^2(\Omega)} \leq
\|u-\phi\|_{L^2(\Omega)} \lesssim \|u-\phi\|_{\Xsigz} < \delta.
\end{equation}
So taking $\delta > 0$ small enough and recalling \eqref{H} and \eqref{L} with $n$ replaced by $Pu$, one finds by \eqref{H} and \eqref{HPphi} that
\begin{equation}\label{eq2}
\left|\J(B\circ Pu)-\J(\phi)\right|^{1-\theta}\leq C \|H'(Pu)\|_{{\mathcal N'}},
\end{equation}
whenever $\|u-\phi\|_{\Xsigz} < \delta$. Then we claim that
 \begin{claim}
Let $\delta > 0$ be small enough. There exists a constant $C \geq 0$ such that
\begin{equation}\label{ls1}
\|H'(Pu)\|_{\mathcal N'} \leq C \|\J'(u)\|_{\Xsigz'}
\end{equation}
for all $u \in \Xsigz$ satisfying $\|u-\phi\|_{\Xsigz} < \delta$. 
 \end{claim}

 \begin{proof}
Note by \eqref{dG} that
$$
\|H'(Pu)\|_{\mathcal N'} \lesssim
\left\| \J'(B \circ P u) \right\|_{\Xsigz'} 
  \left\| B' (P u) \right\|_{\mathcal L(\Xsigz', \Xsigz)}
  \leq C \left\| \J'(B \circ P u) \right\|_{\Xsigz'}
  $$
  for all $u \in \Xsigz$ satisfying $\|u-\phi\|_{\Xsigz} < \delta$ small enough,
since $B'$ is continuous from $U^*$ to $\mathcal L(\Xsigz', \Xsigz)$ (hence, in particular, $B'$ is bounded in a small neighbourhood of $P\phi$) and $\|Pu-P\phi\|_{\Xsigz'} \lesssim \|u-\phi\|_{\Xsigz} < \delta$ by $\dim \mathcal N < \infty$ (see \eqref{PuPphi}). Then
\begin{align*}
\|H'(Pu)\|_{\mathcal N'}
&\leq C \left(
\|\J'(u)\|_{\Xsigz'} 
+ \|\J'(B \circ P u) - \J'(u)\|_{\Xsigz'} \right).
\end{align*}
By the Mean-Value Theorem (see, e.g.,~\cite{GT}), one may also find that
\begin{align*}
\lefteqn{
\|\J'(B \circ P u) - \J'(u)\|_{\Xsigz'} 
}\\
&\leq \sup_{h \in \gamma (B\circ Pu,u)} \|\J''(h)\|_{\mathcal L(\Xsigz, \Xsigz')} 
\left\| B \circ P u - u \right\|_{\Xsigz}\\
&\leq C \left\| B \circ P u - B \circ (\J' + P) u \right\|_{\Xsigz}\\
&\leq C \sup \{ \|B'(Pu+h)\|_{\mathcal L(\Xsigz', \Xsigz)} \colon h \in \gamma(0, \J'(u))\}
 \|\J'(u)\|_{\Xsigz'} \leq C \|\J'(u)\|_{\Xsigz'}.
\end{align*}
Here and henceforth, $\gamma(u,v)$ denotes the line segment connecting $u$ and $v$. Indeed, due to the continuity of $\J'' : \Xsigz \to \mathcal L(\Xsigz, \Xsigz')$ (at $\phi$), we note that $\J''$ is bounded in a small neighbourhood of $\phi$ (in $\Xsigz$). Thus $\J''(h)$ is bounded in $\mathcal L(\Xsigz, \Xsigz')$ for $h \in \gamma(B \circ P u, u)$, since one finds that
\begin{align*}
\|B \circ P u - \phi\|_{\Xsigz}
&= \|B \circ P u - B \circ (\J' + P) \phi\|_{\Xsigz}\\
&\leq \sup \{ \|B'(h)\|_{\mathcal L(\Xsigz', \Xsigz)} \colon h \in \gamma(Pu, P\phi)\} \|Pu - P\phi\|_{\Xsigz'}\\
&\leq C \|u - \phi \|_{\Xsigz}
\end{align*}
(see also \eqref{u-phi}). Moreover, we also used
\begin{equation}\label{B'Pu+h}
\sup \{ \|B'(Pu+h)\|_{\mathcal L(\Xsigz', \Xsigz)} \colon h \in \gamma(0, \J'(u))\} \leq C.
\end{equation}
To see this, we observe that
\begin{align*}
\|Pu + h - P\phi\|_{\Xsigz'}
&\leq \|Pu - P\phi\|_{\Xsigz'} + \|h\|_{\Xsigz'}\\
&\lesssim \|u - \phi\|_{\Xsigz} + \|\J'(u)\|_{\Xsigz'}
\leq \delta + \|\J'(u)\|_{\Xsigz'}
\end{align*}
and
\begin{align*}
\|\J'(u)\|_{\Xsigz'} 
&= \|\J'(u) - \J'(\phi)\|_{\Xsigz'}\\
&\leq \sup_{v \in \gamma(u,\phi)}\|\J''(v)\|_{\mathcal L(\Xsigz,\Xsigz')} \|u - \phi\|_{\Xsigz}
\leq \sup_{v \in \gamma(u,\phi)}\|\J''(v)\|_{\mathcal L(\Xsigz,\Xsigz')} \delta. 
\end{align*}
  Therefore we observe that $Pu + h$ lies on a small neighbourhood of $P\phi$ in $\Xsigz'$ (and also $Pu + h \in U^*$) for $\delta>0$ small enough. 
  Thus \eqref{B'Pu+h} follows from the continuity of $B'$ at $P\phi$. Hence, we finally obtain \eqref{ls1}.
  \end{proof}

We next discuss how to replace $\J(B\circ P u)$ by $\J(u)$ in \eqref{eq2} and how to control an error arising from the replacement. By applying Taylor's theorem to $\J$, one has
\begin{align*}
 \left| \J(B\circ P u) - \J(u) \right|
 &\leq \|\J'(u)\|_{\Xsigz'} \|B \circ P u - u\|_{\Xsigz} \\
 &\quad + \dfrac 1 2 \|\J''(h)\|_{\mathcal L (\Xsigz,\Xsigz')} \|B \circ Pu - u\|_{\Xsigz}^2
\end{align*}
for some $h \in \gamma(B \circ P u, u)$. Then as in \eqref{B'Pu+h} we infer that
\begin{align*}
\|B \circ P u - u\|_{\Xsigz}
&= \|B \circ P u - B \circ (\J' + P) u\|_{\Xsigz}\\
&\leq \sup \{ \|B'(Pu+h)\|_{\mathcal L(\Xsigz', \Xsigz)} \colon h \in \gamma(0, \J'(u))\}
 \|\J'(u)\|_{\Xsigz'}\\
&\leq C \|\J'(u)\|_{\Xsigz'}.
\end{align*}
Thus we have obtained
$$
\left| \J(B \circ P u) - \J(u) \right| \leq C \|\J'(u)\|_{\Xsigz'}^2,
$$
whenever $\|u - \phi\|_{\Xsigz} \leq 1$.

Combining the inequality above with \eqref{eq2} and \eqref{ls1}, we have
\begin{align*}
 |\J(u)-\J(\phi)|
 &\leq \left|\J(u)-\J(B\circ P u)\right|
 + \left|\J(B \circ P u) - \J(\phi)\right|\\
 &\leq C \|\J'(u)\|_{\Xsigz}^2 + C
 \|\J'(u)\|_{\Xsigz'}^{1/(1-\theta)}\\
 &\leq C \|\J'(u)\|_{\Xsigz'}^{1/(1-\theta)},
\end{align*}
whenever $\|u - \phi\|_{\Xsigz} < \delta \leq \delta_0$, since $1/(1-\theta) \leq 2$. Thus we have proved that
\begin{equation}\label{new:LS}
\left| \J(v) - \J(\phi) \right|^{1-\theta}
\leq C \left\| \Asig v + g(v) \right\|_{\Xsigz'},
\end{equation} 
whenever $v \in \Xsigz$ and $\|v-\phi\|_{\Xsigz} < \delta$. Thus \eqref{LSI-2} holds (with $g(\cdot)$ replaced by $\tilde g(\cdot)$ in the cases of (\nii) and (\shi)).

In the cases of (\ichi) and (\san), we applied no replacement of $g(\cdot)$. Hence \eqref{LSI-2} follows directly (for the original $g(\cdot)$). In the cases of (\nii) and (\shi), recalling that
 $$
 \|\phi\|_{L^\infty(\Omega)} < \gamma \quad \mbox{ and } \quad g(s) = \tilde g(s) \ \mbox{ if } \ |s| < \gamma \vee \eta
 $$
 and noting that
\begin{align*}
\J(v) = \tJ(v) \ \mbox{ for }  \ v \in \Xsigz \cap L^\infty(\Omega) \ \mbox{
 satisfying } \ \|v\|_{L^\infty(\Omega)} < \eta 
\end{align*}
(here $\tJ$ denotes the functional $\J$ with $g$ replaced by $\tilde g$), we conclude that \eqref{LSI-2} is satisfied for $v \in \Xsigz \cap L^\infty(\Omega)$ satisfying $\|v\|_{L^\infty(\Omega)} < \eta$. Indeed, since $g(v) = \tilde g(v) \in L^\infty(\Omega)$, we see that
$$
\langle \tJ'(v), w \rangle = \langle \Asig v, w \rangle + \int_\Omega \tilde g(v) w \, \d x = \langle \Asig v, w \rangle + \int_\Omega g(v) w \, \d x
$$
for any $w \in \Xsigz$. This completes the proof. \qed


\section{Proof of Theorem \ref{thm:Loj}}\label{S:pr-conv}

This section provides a proof of Theorem \ref{thm:Loj}. Let $(u,w)$ be a solution of \eqref{eq:fCH}--\eqref{eq:bc} and let $\phi$ be a solution to \eqref{elliptic-f} such that 
$$
u(t_n) \to \phi \ \mbox{ strongly in } \Xsigz
\quad \mbox{ and } \quad \Esi(u(t_n)) \searrow \Esi(\phi)
$$
for some sequence $t_n \to \infty$ (hence $\J(u(t)) \geq \J(\phi)$ for all $t \geq 0$). Then $\phi$ is a critical point of $\J$, that is, $\J'(\phi) = 0$. Assume that one of (i)--(iv) of Theorem \ref{thm:Loj} is satisfied. Then thanks to Theorem \ref{T:LSI-bdd}, there exist constants $\theta \in (0,1/2]$, $\omega, \delta > 0$ such that
 \begin{equation}\label{LSI}
  \left| \Esi(v) - \Esi(\phi) \right|^{1-\theta}
   \leq \omega \left\| \Asig v + g(v) \right\|_{\Xsigz'}
 \end{equation}
 for $v \in \Xsz$ satisfying $\|v - \phi\|_{\Xsigz} < \delta$ (and also $\|v\|_{L^\infty(\Omega)} < \eta$ for the cases (\nii) and (\shi)). As for the cases (\nii) and (\shi), we suppose that $\|u\|_{L^\infty(\Omega \times (0,\infty))} < \eta$. 

Set
 $$
 H(t) := (\Esi(u(t)) - \Esi(\phi))^\theta \geq 0,
 $$
 where $\theta$ is as in \eqref{LSI}. Then we see by \eqref{LSI} that
\begin{align*}
 - \dfrac{\d}{\d t} H(t)
&= - \theta \left( \Esi(u(t)) - \Esi(\phi) \right)^{\theta - 1}
 \dfrac{\d}{\d t} \Esi(u(t))\\
&\stackrel{\eqref{ei1}}{\geq} \theta \left( \Esi(u(t)) - \Esi(\phi) \right)^{\theta - 1}
 \|w(t)\|_{\Xsz}^2\\
&\stackrel{\eqref{LSI}}\geq \omega^{-1} \theta \left\| \Asig u(t) + g(u(t)) \right\|_{\Xsigz'}^{-1} \|w(t)\|_{\Xsz}^2\\
&= \omega^{-1} \theta \left\| w(t) \right\|_{\Xsigz'}^{-1} \|w(t)\|_{\Xsz}^2,
\end{align*}
provided that $\|u(t)-\phi\|_{\Xsigz} < \delta$, where $\delta$ is also given by \eqref{LSI}. Here we note that $\left\| w(t) \right\|_{\Xsigz'}\leq C \|w(t)\|_{\Xsz}$.
Thus we obtain
\begin{equation}\label{ineq-main}
- \dfrac{\d}{\d t} H(t)
\geq \omega^{-1} \theta C^{-1} \|w(t)\|_{\Xsz}
= \omega^{-1} \theta C^{-1} \|\As w(t)\|_{\Xsz'},
\end{equation}
provided that $\|u(t)-\phi\|_{\Xsigz} < \delta$.

Now, we claim that
\begin{equation}\label{cl}
 u(t) \to \phi \quad \mbox{ strongly in } \Xsigz \ \mbox{ as }
  \ t \to \infty
\end{equation}
without taking any subsequence.
Indeed, fix any $\nu \in (0,\delta)$ and set 
$$
s_n := \inf \{s \geq t_n \colon \|u(s)-\phi\|_{\Xsigz} \geq \nu\} \in (t_n, +\infty]
$$
for $n$ large enough. Indeed, $\|u(t_n)-\phi\|_{\Xsigz} < \nu$ for $n$ large enough. Hence we deduce that $t_n < s_n$ for $n$ large enough from the right-continuity of $u$ in the strong topology of $\Xsigz$ (see Theorem \ref{thm:exi}). We shall prove $s_{n_\nu} = +\infty$ for some $n_\nu \in \N$. Then $\|u(s)-\phi\|_{\Xsigz} < \nu$ for all $s \geq t_{n_\nu}$, and hence, \eqref{cl} is proved. We assume on the contrary that $s_n$ is finite for all $n \in \mathbb N$. Then $\|u(t)-\phi\|_{\Xsigz} < \nu < \delta$ for all $t \in [t_n,s_n)$, and moreover, we also remark, by the right-continuity of $u(\cdot)$ in the strong topology of $\Xsigz$, that
\begin{equation}\label{contr}
 \|u(s_n)-\phi\|_{\Xsigz} \geq \nu > 0 \quad \mbox{ for all } \ n \in \N.
\end{equation}
Employing \eqref{ineq-main}, we obtain
\begin{align*}
 \|u(s_n)-\phi\|_{\Xsz'}
&\leq
\int^{s_n}_{t_n} \left\| \partial_t u(\tau)\right\|_{\Xsz'} \, \d \tau
+ \|u(t_n)-\phi\|_{\Xsz'}\\
&\leq \omega \theta^{-1} C
\int^{s_n}_{t_n} - \dfrac{\d}{\d \tau} H(\tau) \, \d \tau
+ \|u(t_n)-\phi\|_{\Xsz'}\\
&\leq - \omega \theta^{-1} C \left( H(s_n) - H(t_n)\right)
+ \|u(t_n)-\phi\|_{\Xsz'}\\
&\leq \omega \theta^{-1} H(t_n) + \|u(t_n)-\phi\|_{\Xsz'}
\to 0.
\end{align*}
 Here we employed Lebesgue's differentiation theorem to ensure the measurability of $t \mapsto (\d/\d t)H(t)$ and that
 $$
 \int^{s_n}_{t_n} \dfrac{\d}{\d \tau} H(\tau) \, \d \tau
 \geq H(s_n)-H(t_n),
 $$
 since $t \mapsto H(t)$ is non-increasing and right-continuous in $[0,\infty)$ and differentiable a.e.~in $(0,\infty)$ (see Theorem \ref{thm:exi}). Thus we deduce that
$$
u(s_n) \to \phi \quad \mbox{ strongly in } \Xsz'
\quad \mbox{ as } \ n \to \infty.
$$
Since $u(s)$ is bounded in $\Xsigz$ for all $s \geq 0$ by \eqref{g4-12} (see also \eqref{hp:g4}), one can take a (not relabeled) subsequence of $(s_n)$ such that
$$
u(s_n) \to \phi \quad \mbox{ weakly in } \Xsigz \mbox{ and strongly in } L^2(\mathbb R^N) \quad \mbox{ as } \ n \to \infty.
$$

Since $\Esi(u(\cdot))$ is nonincreasing, one has
$$
\Esi(u(s_n)) \to \Esi(\phi).
$$
Repeating the same argument as in the proof of Lemma \ref{L:omega}, one can show that
$$
u(s_n) \to \phi \quad \mbox{ strongly in } \Xsigz
$$
by using the (weak) lower semicontinuity argument and the uniform
 convexity of $\Xsigz$. However, this is a contradiction to
 \eqref{contr}. Thus \eqref{cl} follows. This completes the proof. \qed

 \begin{remark}[Rate of convergence]
  {\rm
  Recalling the energy estimates for $H(\cdot)$, one may also estimate the rate of convergence $\J(u(t)) \to \J(\phi)$ as $t \mapsto \infty$. Indeed, we have obtained
  $$
  - \dfrac \d {\d t} H(t)
  \geq \theta \left[H(t)\right]^{(\theta-1)/\theta} \|w(t)\|_{\Xsz}^2.
  $$
  By Theorem \ref{thm:Loj}, there exists $t_0 > 0$ such that $\|u(t)-\phi\|_{\Xsigz} < \delta$ for all $t \geq t_0$. Hence 
  $$
  \|w(t)\|_{\Xsz}^2 \geq C \|w(t)\|_{\Xsigz'}^2
  = C \|\Asig u(t) + g(u(t))\|_{\Xsigz'}^2
  \stackrel{\eqref{LSI}}{\geq} C \left[H(t)\right]^{2(1-\theta)/\theta}.
  $$
  Thus we have
  $$
  \dfrac{\d}{\d t} H(t) \leq - C \left[ H(t) \right]^{-\frac{\theta-1}\theta} \quad \mbox{ for all } \ t \geq t_0,
  $$
  whence follows
  $$
  H(t) \leq \left[ H(t_0)^{-(1-2\theta)/\theta} + C(1 - 2\theta)\theta^{-1}  (t-t_0) \right]^{- \theta/(1 - 2\theta)} \ \mbox{ if } \ \theta \in (0,1/2)  
  $$
  and
  $$
  H(t) \leq H(t_0) e^{-C(t-t_0)} \ \mbox{ if } \ \theta = \frac 1 2
  $$
  for all $t \geq t_0$.
  }
  \end{remark}


\section*{Acknowledgements}

Authors are supported by the JSPS-CNR bilateral joint research project: \emph{Innovative Variational Methods for Evolution Equations} and they would also like to
acknowledge the kind hospitality of the Erwin Schr\"odinger International Institute for Mathematics and Physics, where a part of this research was developed 
under the frame of the Thematic Program {\it Nonlinear Flows}. GA is supported by JSPS KAKENHI Grant Number 16H03946 and by the Alexander von Humboldt 
Foundation and by the Carl Friedrich von Siemens Foundation.
The present paper also benefits from the support of the MIUR-PRIN Grant 2010A2TFX2 ``Calculus of Variations''
for AS and GS, and of the GNAMPA (Gruppo Nazionale per l'Analisi Matematica, la Probabilit\`a 
e le loro Applicazioni) of INdAM (Istituto Nazionale di Alta Matematica). 
%


\appendix


\section{Poincar\'e type inequality}

The following inequality is actually well known (it may follow from Theorem~1 in~\cite{MaSh}
and interpolation). However, for the convenience of the reader, we give a direct and elementary proof. 

\begin{propo}[Fractional Poincar\'e inequality]\label{apdx:P:poincare}
Let $0 < s < 1$ and $1 \leq p < \infty$. Then there is a constant $c_{P}$ depending only on $p$, $s$, $N$ and the diameter of $\Omega$ such that
\begin{equation}\label{eq:poincare-p}
  \| v\|^p_{L^p(\Omega)}\le c_P \iint_{\RdN} \dfrac{|v(x)-v(y)|^p}{|x-y|^{N+ps}} \, \dx \, \dy
\end{equation}
for all $v\in W^{s,p}(\RN)$ satisfying $v = 0$ a.e.~in $\RN \setminus \Omega$.
\end{propo}

\begin{proof}
Let $v\in W^{s,p}(\RN)$ be such that $v = 0$ a.e.~in $\RN \setminus \Omega$.
 One can take $R > 0$ such that $\Omega$ is contained in the open ball $B_R$ of radius $R$ centered at the origin. Then, by definition of the Gagliardo-seminorm, we see that
\begin{align*}
 [v]_{W^{s,p}(\mathbb{R}^N)}^p &\geq \int_{\Omega^c} \left(\int_\Omega
 \dfrac{|v(y)|^p}{|x-y|^{N+sp}} \, \dy \right)\, \dx\\
&\geq \int_\Omega \left( \int_{\Omega^c \cap B_{R+1}} \dfrac 1 {|x-y|^{N+sp}} \, \dx \right) |v(y)|^p \, \dy\\
&\geq \dfrac{|B_{R+1} \setminus B_R|}{(2R+1)^{N+sp}} \|v\|_{L^p(\Omega)}^p,
\end{align*}
where $\Omega^c$ stands for the complement of $\Omega$ and $|B_{R+1}
\setminus B_R|$ denotes the Lebesgue measure of the set $B_{R+1}
\setminus B_R$. Note that $|B_{R+1} \setminus \Omega| \geq |B_{R+1}
\setminus B_R| > 0$. Thus \eqref{eq:poincare-p} follows.
\end{proof}


\section{Some technical details}


\subsection{Proof of \eqref{exi:12}}\label{apdx:s:beta-est}

We derive \eqref{exi:12} by employing an approximation argument (see also~\cite{ASSe1}). Let $u$ be a solution and let $\beta_\vep$ be the
Yosida approximation of $\beta$. Test \eqref{eq:u} by $\beta_\vep(u)$ instead of $\beta(u)$. Here we observe that $\beta_\vep(u)$ belongs 
to $\Xsigz$ due to the Lipschitz continuity of $\beta_\vep$ and $u \in \Xsigz$, and hence,
\begin{align*}
\left\langle \Asig u , \beta_\vep(u) \right\rangle_{\Xsigz}
 = \dfrac{C_\sigma}2 \iint_{\RdN} \dfrac{\left(u(x,t)-u(y,t)\right)\left(\beta_\vep(u(x,t))-\beta_\vep(u(y,t))\right)}{|x-y|^{N+2\sigma}} \, \dx \, \dy
 \geq 0
\end{align*}
by the monotonicity of $\beta_\vep$. Thus recalling that it has already been
proved that $\beta(u) \in L^2(0,T;L^2(\Omega))$, we find that
\begin{align*}
 \left( \beta(u), \beta_\vep(u) \right)
 &\leq \left( w + \lambda u, \beta_\vep(u) \right).
\end{align*}
Note that $\|\beta_\vep(u)\|_{L^2(\Omega)}^2 \leq (\beta(u), \beta_\vep(u))$. Then we infer that
$$
\int^{t+1}_t \left\|\beta_\vep(u(\tau))\right\|_{L^2(\Omega)}^2 \, \d \tau
\leq \int^{t+1}_t \left( w(\tau) + \lambda u(\tau), \beta_\vep(u(\tau)) \right) \, \d \tau.
$$
Hence passing to the limit as $\vep \to 0_+$ and exploiting the fact that
$$
\beta_\vep(u) \to \beta(u) \quad \mbox{ strongly in } L^2(t,t+1;L^2(\Omega)) \quad \mbox{ for } \ t > 0
$$
by Lebesgue's dominated convergence theorem (recall that $|\beta_\vep(r)| \leq |\beta(r)|$ and $\beta(u) \in L^2(t,t+1;L^2(\Omega))$), we obtain \eqref{exi:12}.


\subsection{Proof of \eqref{b-leq-f}}\label{apdx:ss:bu}

Let $p \geq 2$ and set $\delta(s) := |\beta(s)|^{p-2} \beta(s)$. Then $\delta$ is maximal monotone in $\R$. Denote by $j_\vep^\delta$
and $\delta_\vep$ the resolvent and Yosida approximation, respectively, of $\delta$ (more precisely, for each $\vep > 0$, $j_\vep^\delta$ 
is the inverse of the map $s \mapsto s + \vep \delta(s)$ and $\delta_\vep(s) := (s - j_\vep^\delta(s))/\vep$). Since $j_\vep^\delta$ is non-expansive, 
i.e., $|j_\vep^\delta(s) - j_\vep^\delta(\sigma)| \leq |s - \sigma|$ for $s, \sigma \in \R$, we note that $j_\vep^\delta(u) \in \Xsz$ if $u \in \Xsz$, 
and hence, so does $\delta_\vep(u)$. Test \eqref{eq:u:ell} by $\delta_\vep(u) \in \Xsigz$ to have
\begin{equation}\label{Asu-del}
\left\langle \Asig u, \delta_\vep(u) \right\rangle_{\Xsigz} + \int_\Omega \beta(u) \delta_\vep(u) \, \d x = \int_\Omega f \delta_\vep(u) \, \d x.
\end{equation}
We then note that
$$
\beta(u) \delta_\vep(u) = \beta(u) |\beta(j_\vep^\delta u)|^{p-2} \beta(j_\vep^\delta u) \geq |\beta(j_\vep^\delta u)|^p,
$$
since one can write $\delta_\vep (s) = \delta(j_\vep^\delta s)$ for $s \in \R$. Here we also used the fact that
$$
\beta(s) \beta(j_\vep^\delta s) \geq \beta(j_\vep^\delta s)^2 \quad \mbox{ for all } \ s \in \R,
$$
thanks to the monotonicity of $\beta$ and $\beta(0) = 0$ along with the following properties of $j_\vep^\delta$:
$$
0 \leq \pm j_\vep^\delta(s) \leq \pm s \quad \mbox{ if } \ \pm s \geq 0.
$$
Moreover, we remark that
$$
\left\langle \Asig u, \delta_\vep(u) \right\rangle_{\Xsigz} \geq 0 \quad \mbox{ for } \ u \in \Xsigz
$$
by the monotonicity of $\delta_\vep$. Thus combining these facts with \eqref{Asu-del}, we obtain
$$
\|\beta(j_\vep^\delta u)\|_{L^p(\Omega)}^p \leq \|f\|_{L^p(\Omega)} \|\beta(j_\vep^\delta u)\|_{L^p(\Omega)}^{p-1},
$$
which implies
$$
\|\beta(j_\vep^\delta u)\|_{L^p(\Omega)} \leq \|f\|_{L^p(\Omega)}.
$$
Letting $\vep \to 0_+$, one can conclude by Fatou's lemma that
$$
\beta(u) \in L^p(\Omega) \quad \mbox{ and } \quad
\|\beta(u)\|_{L^p(\Omega)} \leq \|f\|_{L^p(\Omega)}
$$
since $j_\vep^\delta(s) \to s$ for $s \in \R$.

\subsection{Proof of Lemma \ref{L:ellip}}\label{apdx:ss:L:ellip}

Let $p \geq 2$. Set $\gamma^p(s) := |s|^{p-2}s$. Then $\gamma$ is maximal monotone in $\R$. As in the last subsection, 
set the resolvent $j^p_\mu(s) := (1 + \mu \gamma^p)^{-1} (s)$ and the Yosida approximation $\gamma^p_\mu (s) := (s - j^p_\mu (s))/\mu = \gamma^p(j^p_\mu(s))$ 
(for $s \in \mathbb R$) of $\gamma$. We further remark that $\gamma^p_\mu(0) = j^p_\mu(0) = 0$. Then $j^p_\mu v(\cdot)$ and $\gamma^p_\mu(v(\cdot))$ belong to $\Xsigz$ if $v \in \Xsigz$.
%
Assume that $u \in \Xsigz \cap L^p(\Omega)$ and test \eqref{eq:u:ell} by $\gamma^p_\mu(u) \in \Xsigz$ to get
\begin{align*}
\dfrac{C_\sigma}2 \iint_{\mathbb R^{2N}}
\dfrac{\left(u(x)-u(y)\right)\left(\gamma^p_\mu(u(x)) - \gamma^p_\mu(u(y)) \right)}{|x-y|^{N+2\sigma}} ~\dx \, \dy
+ \int_\Omega \beta(u) \gamma^p_\mu(u) \,\dx\\
= \int_\Omega f \gamma^p_\mu(u) \,\dx 
\leq \|j_\mu^p (u)\|_{L^p(\Omega)}^{p-1} \|f\|_{L^p(\Omega)}.
\end{align*}
In the procedure above, we have also used that
$$
\|\gamma^p_\mu(u)\|_{L^{p'}(\Omega)}
= \|j^p_\mu(u)\|_{L^p(\Omega)}^{p-1}.
$$
Here we note by $\beta(0) = 0$ that
$$
\int_\Omega \beta(u) \gamma^p_\mu(u) ~\d x
\geq 0
$$
and that
\begin{align*}
\lefteqn{
\iint_{\mathbb R^{2N}}
\dfrac{\left(u(x)-u(y)\right)\left(\gamma^p_\mu(u(x)) - \gamma^p_\mu(u(y)) \right)}{|x-y|^{N+2\sigma}} \, \dx \, \dy
}\\
&\geq
\iint_{\mathbb R^{2N}}
\dfrac{\left(j^p_\mu (u(x))-j^p_\mu (u(y))\right)\left(\gamma^p_\mu(u(x)) -
 \gamma^p_\mu(u(y)) \right)}{|x-y|^{N+2\sigma}} \, \dx \, \d y\\
&\geq
\omega_0 \iint_{\mathbb R^{2N}}
\dfrac{\big|j^p_\mu (u(x))-j^p_\mu (u(y)) \big|^p}{|x-y|^{N+2\sigma}} \, \dx \, \dy
\geq \omega_0 [j^p_\mu(u)]_{W^{\frac{2\sigma}p, p}(\RN)}^p
\end{align*}
by using the well-known inequality,
$$
\omega_0 |a-b|^p \leq \left(a-b\right)\left(|a|^{p-2}a - |b|^{p-2}b\right)
\quad \mbox{ for all } \ a,b \in \mathbb R
\ \mbox{ and } \ p \geq 2
$$
for some constant $\omega_0 > 0$.
%
%
Combining these facts and using the monotonicity of $\gamma^p_\mu$, we obtain
$$
\dfrac{C_\sigma}2 \omega_0 [j^p_\mu(u)]_{W^{\frac{2\sigma}p, p}(\RN)}^p
\leq \|j_\mu^p(u)\|_{L^p(\Omega)}^{p-1} \|f\|_{L^p(\Omega)}.
$$
Hence by virtue of the Poincar\'e type inequality (see Proposition \ref{apdx:P:poincare} and recall that $j_\mu^p(u) \equiv 0$ in $\R^N \setminus \Omega$), we have
 $$
 [j^p_\mu(u)]_{W^{\frac{2\sigma}p,p}(\RN)} \leq C \|f\|_{L^p(\Omega)}.
 $$
Passing to the limit as $\mu \to 0$, since $j^p_\mu u \to u$ strongly in
 $L^2(\mathbb R^N)$, we deduce that $u \in W^{\frac{2\sigma}p, p}(\RN)$ and that
\begin{equation}\label{regu}
\|u\|_{W^{\frac{2\sigma}p, p}(\RN)}
\leq C \|f\|_{L^p(\Omega)}.
\end{equation}
This completes the proof.


\subsection{Proof of \eqref{star}}\label{apdx:ss:star} From the weak formulation,
\begin{align*}
 \left\langle \As v , z \right\rangle_{\Xsz} &= \dfrac{C_s}2 \iint_{\RdN} \dfrac{\left(v(x)-v(y)\right)\left(z(x)-z(y)\right)}{|x-y|^{N+2s}} \, \dx \, \dy\\
 &\leq \|v\|_{\Xsz} \|z\|_{\Xsz} \quad \mbox{ for all } \ v,z \in \Xsz.
\end{align*}
On the other hand, we see that
$$
\left\langle \As v , v \right\rangle_{\Xsz} = \|v\|_{\Xsz}^2 \quad \mbox{ for all } \ v \in \Xsz.
$$
Thus
$$
\|\As v\|_{\Xsz'} = \|v\|_{\Xsz} \quad \mbox{ for all } \ v \in \Xsz,
$$
which is \eqref{star}.


\subsection{Proof of \eqref{Lu1=h1}}\label{apdx:ss:Lu1}

From the fact that $h_1 \in \mathcal N^\bot \cap L^p(\Omega) \subset \Xsigz'$, by Claim \ref{new:claim1}, there exists $u_1 \in \Xsigz$ such that
 $$
 \left( \mathcal L(\phi) + P \right)(u_1) = h_1.
 $$
 Test it by $v \in \mathcal N$. Then
 $$
 \left( \mathcal L(\phi)(u_1), v \right) + \left( P u_1, v \right)
 = (h_1,v),
 $$
 which implies $(Pu_1, v) = 0$ for all $v \in \mathcal N$ by $h_1 \in \mathcal N^\bot$ and the symmetry of $\mathcal L(\phi)$. 
 Hence $Pu_1\in {\mathcal N}^\bot$, and therefore, $P u_1 = 0$ (i.e., $u_1 \in \mathcal N^\bot$) and $\mathcal L(\phi) u_1 = h_1$.
 Since $h_1 \in L^p(\Omega)$, one deduces that $u_1 \in \Xp$ by definition.

\subsection{Analyticity of $\J$ in $V_p$}\label{apdx:s:analytic}

We have already checked that $\J' : \Xp \to L^p(\Omega)$ is analytic in $V_p$. Hence, for $n \in \N$ there exist a constant $r > 0$ and a symmetric bounded $n$ linear form $T_n : (\Xp)^n \to L^p(\Omega)$ such that, for $v \in V_p$ and $h \in \Xp$,
\begin{align}
 \J(v+h)-\J(v)
 &= \int^1_0 \dfrac{\d}{\d t} \J(v+th) \, \d t\nonumber\\
 &= \int^1_0 \left\langle \J'(v+th),h\right\rangle_{\Xsigz} \, \d t\nonumber\\
 &= \int^1_0 \bigg\langle h, \sum_{n=0}^\infty t^n [T_n(v)] (\underbrace{h,\ldots,h}_{n \text{ times}})  \bigg\rangle_{L^p(\Omega)} \d t \label{apdx:series}\\
 &= \sum_{n=0}^\infty \bigg\langle h, \dfrac{1}{n+1} [T_n(v)](\underbrace{h,\ldots,h}_{n \text{ times}}) \bigg\rangle_{L^p(\Omega)}\nonumber
\end{align}
provided that $\|h\|_{\Xp} < r$. Here, we note that, for all $t \in [0,1]$,
\begin{align*}
\sum_{n=0}^\infty t^n \Big\| [T_n(v)] (\underbrace{h,\ldots,h}_{n \text{ times}}) \Big\|_{L^p(\Omega)}
 &\leq \sum_{n=0}^\infty \|T_n(v)\|_{\mathcal L^n(\Xp,L^p(\Omega))} \|h\|_{\Xp}^n < \infty,
\end{align*}
provided that $\|h\|_{\Xp} < r$. Hence the series \eqref{apdx:series} is convergent uniformly for $t \in [0,1]$, and therefore, the termwise integration is admissible. Set
$$
\widehat T_n(v)(h_1,h_2,\ldots,h_n) := \dfrac 1 {n+1} \bigg\langle h_n, [T_{n-1}(v)] (h_1,\ldots,h_{n-1}) \bigg\rangle_{L^p(\Omega)}.
$$
Then,  repeating the same argument, one deduces that
\begin{align*}
\sup_{n \in \N} \|\widehat T_n(v)\|_{\mathcal L^n(\Xp, L^p(\Omega))} r^n
 &\leq \sup_{n \in \N}\dfrac{Cr}{n+1} \left\|T_{n-1}(v)\right\|_{\mathcal L^{n-1}(\Xp,L^p(\Omega))} r^{n-1} < \infty.
\end{align*}
Consequently, $\J$ turns out to be analytic on $V_p$.

\section{Hopf's lemma for the fractional Laplacian}\label{S:hopf}

Let us state Hopf's lemma provided in~\cite{GrSe} with slight and straightforward modifications. 

\begin{propo}[Hopf's lemma for the fractional Laplacian~\cite{GrSe}]\label{P:Hopf}
 Let us assume that $\Omega \subset \R^N$ satisfies the \emph{uniform interior ball condition}, that is, there exists $r > 0$ such that for all $x \in \partial \Omega$ one can take a ball $B_x \subset \Omega$ of radius $r$ such that $\partial B_x \cap \partial \Omega = \{x\}$. Let $c \in L^\infty(\Omega)$ and let $u$ be a lower semicontinuous function $u : \R^N \to \R$ satisfying
 $$
 \Ds u(x) \geq c(x) u(x) \ \mbox{ a.e.~in } \Omega.
 $$
 If $u \geq 0$ in all of $\R^N$, then either $u$ vanishes identically in $\Omega$ or there exists $\delta_0 > 0$ such that, for any $x \in \partial \Omega$,
 \begin{equation}\label{hopf}
  \liminf_{B_x \ni z \to x} \dfrac{u(z)}{\delta(z)^s} \geq \delta_0,
 \end{equation}
 where $\delta(z)$ is given by
 $$
 \delta(z) := 
 \mathrm{dist} (z, \partial B_x)
 \ \mbox{ for } z \in B_x.
 $$
\end{propo}

\begin{remark}\label{R:Hopf}
 {\rm
The conclusion of the proposition above also holds true if one assumes that $u \geq 0$ in $\R^N \setminus \Omega$ and $c \leq 0$ in $\Omega$ (instead of $u \geq 0$ in $\R^N$).
 }
\end{remark}


\section{Justification of the proof for Theorem \ref{thm:ener}}\label{apdx:thm3}

Let $T > 0$, $N \in \N$ and set $\tau := T/N > 0$. As in~\cite{ASSe1} (see also \S \ref{ss:ex}), we introduce the following time-discretization of \eqref{eq:w} and \eqref{eq:u}:
\begin{alignat}{4}
& \dfrac{u_n - u_{n-1}}{\tau} + \As w_n = 0 \ &\mbox{ in }&
 \Xsz',\label{td1}\\
& w_n = \Asig u_n + \beta(u_n) - \lambda u_{n-1} \ &\mbox{ in }&
 \Xsigz'\label{td2}
\end{alignat}
for $n = 1,2,\ldots,N$ (here $\beta$ may be replaced by $\beta_\vep$ if necessary as in \S \ref{ss:ex}). Then as in~\cite{ASSe1}, one obtains
\begin{equation}\label{td-e1}
 \sum_{n = 1}^N  \tau \|w_n\|_{\Xsz}^2 + \sum_{n=1}^N \tau \left\| \dfrac{u_n - u_{n-1}}\tau\right\|_{\Xsz'}^2 + \max_n \Esi(u_n) \leq C_1,
\end{equation}
where $C_1$ is a constant depending only on $\Esi(u_0)$ and the constant $C$ of \eqref{g4-12}. We next differentiate \eqref{td2} as follows:
\begin{equation}\label{td2-diff}
 w_n - w_{n-1} = \Asig (u_n - u_{n-1}) + \beta(u_n) - \beta(u_{n-1}) - \lambda \left( u_{n-1} - u_{n-2} \right). 
\end{equation}
Test it by $u_n - u_{n-1}$. It follows that
$$
\left(w_n - w_{n-1}, u_n - u_{n-1}\right)
\geq \|u_n - u_{n-1}\|_{\Xsigz}^2 - \lambda \left( u_{n-1} - u_{n-2}, u_n - u_{n-1}\right).
$$
Moreover, the multiplication of \eqref{td1} with $w_n - w_{n-1}$ implies
$$
\left( \dfrac{u_n - u_{n-1}}\tau, w_n - w_{n-1} \right) + \dfrac 1 2 \|w_n\|_{\Xsz}^2 \leq \dfrac 1 2 \|w_{n-1}\|_{\Xsz}^2.
$$
Thus we find that
\begin{align}
 \dfrac 1 2 \|w_n\|_{\Xsz}^2 +  \tau \left\| \dfrac{u_n - u_{n-1}}\tau \right\|_{\Xsigz}^2
 \leq \dfrac 1 2  \|w_{n-1}\|_{\Xsz}^2 + \lambda \tau \left( \dfrac{u_{n-1} - u_{n-2}} \tau, \dfrac{u_n - u_{n-1}}\tau \right).
 \label{save}
\end{align}
Furthermore, multiplying both sides by $n\tau$, one has
 \begin{align*}
  \lefteqn{
 \dfrac 1 2 n\tau\|w_n\|_{\Xsz}^2 +  n\tau^2 \left\| \dfrac{u_n - u_{n-1}}\tau \right\|_{\Xsigz}^2
  }\\
  &\leq \dfrac 1 2 (n-1) \tau \|w_{n-1}\|_{\Xsz}^2 + \dfrac 1 2 \tau \|w_{n-1}\|_{\Xsz}^2 + \lambda n\tau^2 \left( \dfrac{u_{n-1} - u_{n-2}} \tau, \dfrac{u_n - u_{n-1}}\tau \right).
 \end{align*}
 For any $m \in \N \cap [2,N]$, summing up from $n = 2$ up to $m$, we deduce that
 \begin{align*}
  \lefteqn{
 \dfrac 1 2 m\tau\|w_m\|_{\Xsz}^2 +  \sum_{n = 2}^m n\tau^2 \left\| \dfrac{u_n - u_{n-1}}\tau \right\|_{\Xsigz}^2
  }\\
  &\leq \dfrac 1 2 \tau \|w_1\|_{\Xsz}^2 + \dfrac 1 2 \sum_{n = 2}^m \tau \|w_{n-1}\|_{\Xsz}^2 + \lambda \sum_{n=2}^m n\tau^2 \left( \dfrac{u_{n-1} - u_{n-2}} \tau, \dfrac{u_n - u_{n-1}}\tau \right).
 \end{align*}
 Moreover, by Ehrling's lemma (along with $\Xsigz \hookrightarrow H_0 \simeq H_0' \hookrightarrow \Xsz'$ compactly), for any $\vep > 0$ one can take $C_\vep \geq 0$ such that
  \begin{align*}
   \lefteqn{
   \sum_{n=2}^m n\tau^2 \left( \dfrac{u_{n-1} - u_{n-2}} \tau, \dfrac{u_n - u_{n-1}}\tau \right)}\\
   &\leq \sum_{n=2}^m n\tau^2 \Bigg(
   \vep \left\| \dfrac{u_{n-1}-u_{n-2}}\tau \right\|_{\Xsigz}^2
   + C_\vep \left\| \dfrac{u_{n-1}-u_{n-2}}\tau \right\|_{\Xsz'}^2\\
   &\quad + \vep \left\| \dfrac{u_n-u_{n-1}}\tau \right\|_{\Xsigz}^2
   + C_\vep \left\| \dfrac{u_n-u_{n-1}}\tau \right\|_{\Xsz'}^2
   \Bigg)\\
   &\leq 2 \sum_{n=2}^m (n+1) \tau^2 \left( \vep \left\| \dfrac{u_{n}-u_{n-1}}\tau \right\|_{\Xsigz}^2
   + C_\vep \left\| \dfrac{u_{n}-u_{n-1}}\tau \right\|_{\Xsz'}^2 \right)\\
   &\quad + 2\vep \left\| u_1-u_0 \right\|_{\Xsigz}^2 + 2C_\vep \left\| u_1-u_0 \right\|_{\Xsz'}^2.
  \end{align*}
  Therefore, combining these facts (with $\vep > 0$ small enough so that $2(n+1) \vep \lambda \leq n /2$ for $n \geq 2$), we deduce that
 \begin{align*}
  \lefteqn{
 \dfrac 1 2 m\tau\|w_m\|_{\Xsz}^2 + \dfrac 1 2 \sum_{n = 2}^m n\tau^2 \left\| \dfrac{u_n - u_{n-1}}\tau \right\|_{\Xsigz}^2
  }\\
  &\leq \sum_{n = 2}^m \tau \|w_{n-1}\|_{\Xsz}^2 + C \Bigg( \sum_{n=2}^m (n+1)\tau^2 \left\| \dfrac{u_{n}-u_{n-1}}\tau \right\|_{\Xsz'}^2 \\
  & \mbox{}~~~~~ + \left\| u_1-u_0 \right\|_{\Xsigz}^2 + \left\| u_1-u_0 \right\|_{\Xsz'}^2 \Bigg)\\
  &\stackrel{\eqref{td-e1}}{\leq} C_2 (1 + m\tau),
 \end{align*}
 where we have also used \eqref{td-e1} and 
 $C_2$ is independent of $m$, $\tau$, $N$ and $T$.

 Now, using the above relation and \eqref{td-e1} again, for any $k \in \N \cap (2,N)$, one can take $n_k \in [2,k]$ such that
 $$
 \|w_{n_k}\|_{\Xsz}^2 + n_k \tau \left\|\dfrac{u_{n_k}-u_{n_{k-1}}}\tau\right\|_{\Xsigz}^2 \leq \dfrac{C_3}{(k-2)\tau},
 $$
 where $C_3 := 2C_2(1+k\tau) + C_1$. Recalling \eqref{save}, summing up both sides from $n = n_k + 1$ until $n = m$ and repeating the same argument as before, we then have
 \begin{align*}
  \lefteqn{
  \dfrac 1 2 \|w_m\|_{\Xsz}^2 + \dfrac 1 2 \sum_{n = n_k + 1}^m \tau \left\|\dfrac{u_n-u_{n-1}}\tau\right\|_{\Xsigz}^2
  }\\
  &\leq \dfrac 1 2 \|w_{n_k}\|_{\Xsz}^2 + \lambda C_\vep \sum_{n = 1}^N \tau \left\|\dfrac{u_n - u_{n-1}}\tau\right\|_{\Xsz'}^2 + \lambda \vep \tau \left\|\dfrac{u_{n_k}-u_{n_k-1}}\tau\right\|_{\Xsigz}^2\\
  &\leq \dfrac{C_3}{(k-2)\tau} + \lambda C_\vep C_1 + \lambda \vep \dfrac{C_3}{2(k-2)\tau}.
 \end{align*}
 Let $t_0 > 0$ be fixed and take $k \in \N$ such that $(k-1)\tau < t_0 \leq k\tau $. Hence, it follows by $C_3 = 2C_2(1+k\tau) + C_1$ that, for any $m \in \N \cap (k+1,N]$,
 \begin{align*}
  \dfrac 1 2 \|w_m\|_{\Xsz}^2 + \dfrac 1 2 \sum_{n = k + 1}^m \tau \left\|\dfrac{u_n-u_{n-1}}\tau\right\|_{\Xsigz}^2
  \leq \dfrac{C_3}{t_0-2\tau} + \lambda C_\vep C_1 + \lambda \vep \dfrac{C_3}{2(t_0-2\tau)}\\
  \leq \dfrac{4C_2(1+t_0+\tau)+2C_1}{t_0} + \lambda C_\vep C_1 + \lambda \vep \dfrac{2C_2(1+t_0+\tau)+C_1}{t_0}
 \end{align*}
 for $\tau > 0$ small enough (so that $t_0 - 2\tau > t_0/2$). Here we recall again that $C_1$, $C_2$, $\vep$, $C_\vep$ are independent of $m$, $\tau$, $N$, $T$ and $t_0$.
  
 Now, recall the piecewise constant and linear interpolants of $(w_n)$ and $(u_n)$, in particular,
 $$
 \bar w_\tau(t) := w_n, \quad u_\tau(t) := \dfrac{t_n-t}\tau u_{n-1} + \dfrac{t-t_{n-1}}\tau u_n \quad \mbox{ for } \ t \in [t_{n-1}, t_n),
 $$
 where $t_n := n\tau$, for $n = 1,2,\ldots,N$. Hence there exists a constant $C_0 \geq 0$ independent of $t_0$, $\tau$, $N$ and $T$ such that
 $$
 \|\bar w_\tau(t)\|_{\Xsz}^2 + \int^t_{t_0+2\tau} \|\partial_r u_\tau(r)\|_{\Xsigz}^2 \, \d r \leq C_0\left(1 + t_0^{-1}\right) \quad \mbox{ for all } \ t \geq t_0 + 2\tau \ \mbox{ and } \ t_0 > 0.
 $$
Fix $\delta > t_0$. Then it follows that
 \begin{alignat*}{4}
  \bar w_\tau &\to w \quad &&\mbox{ weakly star in } L^\infty(\delta,T;\Xsz),\\
  \partial_t u_\tau &\to \partial_t u \quad &&\mbox{ weakly in } L^2(\delta,T;\Xsigz)
 \end{alignat*}
 as $\tau \to 0_+$. From the arbitrariness of $T > 0$ and the fact that $C_0$ is independent of $T$, one concludes that $w \in L^\infty(t_0,\infty;\Xsz)$, $\partial_t u \in L^2(t_0,\infty;\Xsigz)$ and  
 $$
 \|w(t)\|_{\Xsz}^2 + \int^t_{t_0} \|\partial_r u(r)\|_{\Xsigz}^2 \, \d r \leq C_0\left(1 + t_0^{-1}\right) \quad \mbox{ for any } \ t \geq t_0.
 $$
Thus \eqref{reg:13} follows.

\end{document}